\documentclass[11pt,a4paper,american]{article}
\usepackage{babel}
\usepackage[T1]{fontenc}
\usepackage[utf8]{inputenc}
\usepackage{fullpage}
\usepackage[babel]{microtype}
\usepackage{csquotes}
\usepackage{authblk}
\usepackage{titlesec}
\usepackage[giveninits,sortcites,maxbibnames=10,style=ext-numeric]{biblatex}
\usepackage{hyperref}
\usepackage{enumitem}
\usepackage{amsmath}
\usepackage{mathtools}
\usepackage{amsthm}
\usepackage{thmtools}
\usepackage{thm-restate}
\usepackage{dsfont}
\usepackage{amssymb}
\usepackage{diffcoeff}
\usepackage{zref-clever}
\usepackage{xcolor}
\usepackage{mleftright}

\hypersetup{
    hypertexnames = false,
    pdfencoding = auto,
    bookmarksnumbered = true,
    colorlinks = true,
    linkcolor = blue,
    citecolor = teal,
    pdfauthor = {Marco Pozza},
    pdftitle = {Lax–Oleinik formula for nonautonomous Hamilton–Jacobi equations on networks},
    pdfkeywords = {time-dependent Hamilton–Jacobi equations;embedded networks;nonautonomous Hamiltonians;viscosity solutions;comparison principle;representation formulas},
    pdfsubject = {MSC(2020): 35F21, 35R02, 35B51, 49L25},
}

\titlespacing{\paragraph}{0pt}{1em}{.7em}

\declaretheorem[style=plain,parent=section,title=Theorem]{theo}
\declaretheorem[style=plain,sibling=theo,title=Proposition]{prop}
\declaretheorem[style=plain,sibling=theo,title=Corollary]{cor}
\declaretheorem[style=plain,sibling=theo,title=Lemma]{lem}
\declaretheorem[style=definition,sibling=theo,title=Definition]{defin}
\declaretheorem[style=remark,sibling=theo,title=Remark]{rem}

\newlist{Hassum}{enumerate}{1}
\setlist[Hassum]{label=\textbf{(H\arabic*)},ref=\textnormal{\textbf{(H\arabic*)}},font=\normalfont}
\newlist{myenum}{enumerate}{1}
\setlist[myenum]{label=\textbf{\roman*)},ref=\textnormal{\textbf{(\roman*)}},font=\normalfont}
\newlist{mylist}{itemize}{1}
\setlist[mylist]{label=\textbullet,font=\normalfont}
\newlist{steproof}{enumerate}{2}
\setlist[steproof]{font=\bfseries}
\setlist[steproof,1]{label={Step \arabic*.},ref=\arabic*,wide=0pt,nosep}
\setlist[steproof,2]{label=Step \thesteproofi\alph*.,ref=\thesteproofi\alph*,wide=0pt,nosep}

\zcRefTypeSetup{condenum}{nocap}
\zcRefTypeSetup{equation}{
    Name-sg =,
    name-sg =,
    Name-pl =,
    name-pl =,
    Name-sg-ab =,
    name-sg-ab =,
    Name-pl-ab =,
    name-pl-ab =,
    refbounds = {\textup{(},,,\textup{)}},
    reffont = \normalfont,
}
\zcRefTypeSetup{item}{nocap}
\zcRefTypeSetup{steproof}{nocap}

\zcLanguageSetup{american}{
    cap,
    type = steproof,
    Name-sg = Step,
    name-sg = step,
    Name-pl = Steps,
    name-pl = steps,
    type = condenum,
    Name-sg = Condition,
    name-sg = condition,
    Name-pl = Conditions,
    name-pl = conditions,
}

\zcsetup{
    nocomp,
    rangesep = {--},
    nameinlink = false,
    countertype = {
        theo = theorem,
        prop = proposition,
        cor = corollary,
        lem = lemma,
        defin = definition,
        rem = remark,
        myenumi = item,
        Hassumi = condenum,
        steproofi = steproof,
        steproofii = steproof,
    },
}

\DeclareFieldFormat{titlecase:title}{\MakeSentenceCase*{#1}}

\DeclareMathOperator{\co}{co}

\newcommand{\Crm}{\mathrm{C}}
\newcommand{\Dcal}{\mathcal{D}}
\newcommand{\Ecal}{\mathcal{E}}
\newcommand{\Fcal}{\mathcal{F}}
\newcommand{\ovH}{\overline{H}}
\newcommand{\Hcal}{\mathcal{H}}
\newcommand{\ovL}{\overline{L}}
\newcommand{\whL}{\widehat{L}}
\newcommand{\Lrm}{\mathrm{L}}
\newcommand{\Nds}{\mathds{N}}
\newcommand{\Ocal}{\mathcal{O}}
\newcommand{\ovOcal}{\overline{\Ocal}}
\newcommand{\Qcal}{\mathcal{Q}}
\newcommand{\ovQcal}{\overline{\Qcal}}
\newcommand{\Rds}{\mathds{R}}
\newcommand{\Trm}{\mathrm{T}}
\newcommand{\Vbf}{\mathbf{V}}
\newcommand{\ovW}{\overline{W}}
\newcommand{\ovc}{\overline{c}}
\newcommand{\ovs}{\overline{s}}
\newcommand{\ovt}{\overline{t}}
\newcommand{\ovw}{\overline{w}}
\newcommand{\ovTheta}{\overline{\Theta}}
\newcommand{\whTheta}{\widehat{\Theta}}
\newcommand{\ovalpha}{\overline{\alpha}}
\newcommand{\whalpha}{\widehat{\alpha}}
\newcommand{\ovbeta}{\overline{\beta}}
\newcommand{\whbeta}{\widehat{\beta}}
\newcommand{\ovepsilon}{\overline{\epsilon}}
\newcommand{\whepsilon}{\widehat{\epsilon}}
\newcommand{\wtgamma}{\widetilde{\gamma}}
\newcommand{\ovell}{\overline{\ell}}
\newcommand{\ovvartheta}{\overline{\vartheta}}
\newcommand{\whvartheta}{\widehat{\vartheta}}
\newcommand{\ovxi}{\overline{\xi}}
\newcommand{\myemail}[2][]{\textsuperscript{#1}\href{mailto:#2}{\texttt{#2}}}

\mleftright

\title{Lax--Oleinik Formula for Nonautonomous Hamilton--Jacobi Equations on Networks}
\author{Marco Pozza}
\affil{Link Campus University, Rome, Italy. \textit{Email~address:}~\myemail{m.pozza@unilink.it}}
\date{}

\begin{document}

    \maketitle

    \begin{abstract}
        We provide a Lax--Oleinik-type representation formula for solutions to nonautonomous Hamilton--Jacobi equations posed on networks with a rather general geometry. The networks may possess countably many arcs and allow for the presence of loops. We consider Hamiltonians that are convex and superlinear in the momentum variable, and satisfy a Lipschitz-type condition in the time variable. The representation formula is constructed via an overall Lagrangian that accounts for both the arc-specific dynamics and vertex-based constraints, called flux limiters, which ensure the well-posedness of the problem. We prove that the corresponding action functional admits Lipschitz continuous minimizers without needing to rule out the Zeno phenomenon. Furthermore, we demonstrate that the formula yields the unique solution to the problem even when the flux limiters exceed standard upper bounds.
    \end{abstract}

    \paragraph{2020 Mathematics Subject Classification:} 35F21, 35R02, 35B51, 49L25.

    \paragraph{Keywords:} time-dependent Hamilton--Jacobi equations, embedded networks, nonautonomous Hamiltonians, viscosity solutions, comparison principle, representation formulas.

    \section{Introduction}

    True to the title, the purpose of this paper is to provide a Lax--Oleinik-type representation formula for solutions to nonautonomous Hamilton--Jacobi equations posed on networks with a rather general geometry.

    The relevance of this kind of formula lies in its ability to enhance the qualitative analysis of Hamilton--Jacobi equations on networks, similarly to what happens for manifolds or Euclidean spaces. In this setting, in fact, they essentially come into play in a variety of theoretical constructions, such as weak KAM theory~\cite{Fathi08}, the existence of regular subsolutions~\cite{Bernard07}, homogenization problems~\cite{Souganidis99}, large time behavior of solutions~\cite{Fathi98,DaviniSiconolfi06,Ishii08} and selection principles in the ergodic approximation~\cite{DaviniFathiIturriagaZavidovique16}. Moreover, the representation formula provided in~\cite{PozzaSiconolfi23} for autonomous equations posed on networks plays a central role in many situations: large time behavior of solutions~\cite{Pozza25_2}, homogenization problems~\cite{PozzaSiconolfiSorrentino24}, and numerical schemes for the approximation of both solutions~\cite{CarliniSiconolfi25} and effective Hamiltonians~\cite{CoscettiPozza25}.

    We consider a connected network \(\Gamma\) embedded in \(\Rds^{N}\) with at most countable many \emph{arcs} \(\gamma\), namely regular simple curves with Euclidean length \(\lvert \gamma \rvert\) and parametrized in \([0, \lvert \gamma \rvert]\), linking points of \(\Rds^{N}\) called \emph{vertices}. A Hamiltonian on \(\Gamma\) is a collection of Hamiltonians \(H_{\gamma} \colon [0, \lvert \gamma \rvert] \times [0, \infty) \times \Rds \to \Rds\) indexed by the arcs, depending on state, time and momentum variable, with the crucial feature that Hamiltonians associated to arcs possessing different support are totally unrelated. We are interested in the corresponding family of equations
    \begin{equation}\label{eq:introprob}
        \difcp{U}{t} + H_{\gamma}(s, t, \difcp{U}{s}) \qquad \text{in \((0, \lvert \gamma \rvert) \times (0, \infty)\)},
    \end{equation}
    and a solution on \(\Gamma\) is a continuous function \(u \colon \Gamma \times [0, \infty) \to \Rds\) such that \(u(\gamma(s), t)\) solves~\zcref{eq:introprob}, in the viscosity sense, for every arc \(\gamma\) and satisfies additional conditions on the discontinuity interfaces
    \begin{equation}\label{eq:discinterf}
        \{(x, t), \, t \in [0, \infty)\} \qquad \text{with \(x \in \Vbf\)},
    \end{equation}
    where \(\Vbf\) denotes the set of vertices. It has been established in~\cite{ImbertMonneau17} that, to get existence and uniqueness of solutions, \zcref{eq:introprob} must be coupled not only with a continuous initial datum at \(t = 0\), but also with a \emph{flux limiter}, that is the choice of appropriate maps \(c_{x} \colon [0, \infty) \to \Rds\) for each \(x \in \Vbf\). We also report the contribution of~\cite{LionsSouganidis17,Morfe20}, where the time-dependent problem is studied in junctions with Kirchoff-type Neumann conditions at vertices.

    The notion of solution employed here is an extension of the one given in~\cite{Siconolfi22} for the autonomous case, where flux limiters appear in the condition the solution must satisfy on the discontinuity interfaces~\zcref{eq:discinterf} and, among other things, bound from above the time derivative of subsolutions on vertices. Even if an initial datum is fixed, solutions can change according to the choice of flux limiter so that they must be taken into account in representation formulas.

    It has been proved in~\cite{CarliniCoscettiPozza25} that, for the autonomous equation, this definition of solution is equivalent to the one introduced in~\cite{ImbertMonneau17}. While we expect this equivalence to hold in the nonautonomous setting, its proof is outside the scope of the present work.

    We assume that the \(H_{\gamma}\) are continuous on all their arguments plus convex and superlinear in the momentum variable, which are usual conditions to ensure the validity of the Lax--Oleinik formula for time-dependent problems posed on manifolds and Euclidean settings. Under these conditions it is natural to define the Lagrangians \(L_{\gamma}\) as the convex conjugates of the Hamiltonians. We further define an overall Lagrangian \(L\) on the whole tangent bundle \(\Trm \Gamma\), playing the role of integrand in the Lax--Oleinik formula, by gluing together the \(L_{\gamma}\) and taking into account the chosen flux limiter \(c_{x}\) setting
    \begin{equation*}
        L(x, q, t) = \frac{\lvert q \rvert^{2}}{2} + c_{x}(t) \qquad \text{for any \(x \in \Vbf\)}.
    \end{equation*}
    This Lagrangian is convex in the speed variable, but it may be discontinuous at vertices and, in general, is not lower semicontinuous. Nevertheless, we show in \zcref{actlsc} that the action
    \begin{equation*}
        \xi \longmapsto \int_{t_{1}}^{t_{2}} L \left( \xi(\tau), \dot{\xi}(\tau), \tau \right) \dl \tau,
    \end{equation*}
    defined on the set of absolutely continuous curves from \([t_{1}, t_{2}]\) into \(\Gamma\), is lower semicontinuous.

    To ensure the existence of Lipschitz continuous minimizers for the action we also require a Lipschitz continuity condition on the time variable for both the Hamiltonians and the flux limiter, see \zcref[nosort]{condHtlip,condfllip}. We point out that, in the manifolds/Euclidean setting, a classical alternative condition to obtain Lipschitz continuous minimizers is Lipschitz continuity with respect to the state variable, which is not suitable for our problem due to the discontinuity of the overall Lagrangian at vertices.

    Our work extends the representation formula provided in~\cite{PozzaSiconolfi23} for solutions to autonomous equations on networks to the nonautonomous case, while introducing several improvements. First, we no longer need to rule out Zeno phenomena in the minimization of the action functional---namely, the possibility that candidate minimizers wildly oscillate around a discontinuity interface---which was a focal point of the aforementioned paper. Nonetheless, since this phenomenon is notoriously disruptive for both analysis and simulations, we provide a criterion to rule it out in \zcref{sec:nozeno} motivated by potential future applications. \\
    Another improvement concerns the flux limiter. In~\cite{PozzaSiconolfi23} it was the choice of a constant \(c_{x}\) for every \(x \in \Vbf\) which must satisfy an upper bound in order to avoid the Zeno phenomenon. Here each \(c_{x}\) is a time-dependent function and, while we initially impose an upper bound, we show in \zcref{sec:beyondfl} that our representation formula still provides the unique solution to the problem even for flux limiters that do not respect this bound. \\
    Furthermore, we consider a network under more general conditions: the arcs can be countable many rather than finite, and \emph{loops}---arcs with initial and final points coinciding---are admitted. We emphasize that in the known literature about time-dependent Hamilton--Jacobi equations on networks no loop are admitted, see, e.g., \cite{ImbertMonneau17,Siconolfi22}.

    Previous contributions on the same topic we deal with include~\cite{ImbertMonneauZidani12}, where the results are given in the case of junctions (i.e., networks with a single vertex), and~\cite{IturriagaSanchezMorgado17}, whose focus is on the relationship between the Lax--Oleinik formula and weak KAM theory and the connection with time-dependent problems is not taken into account. In both cases the Hamiltonians are assumed to be autonomous. \\
    Additionally, a finite horizon deterministic mean field game on networks with finitely many arcs is analyzed in~\cite{AchdouMannucciMarchiTchou24}. The authors prove that the value function of the related optimal control problem is a solution to a nonautonomous equation with Hamiltonians of the form
    \begin{equation*}
        H_{\gamma}(s, t, \mu) = \frac{\lvert \mu \rvert^{2}}{2} - \ell_{\gamma}(s, t),
    \end{equation*}
    where \(\ell_{\gamma}\) is a bounded continuous function.

    The paper is organized as follows: \zcref{sec:preliminaries} provides some basics facts regarding networks, Hamiltonians defined on it, and presents the main hypothesis. \zcref[S]{sec:minact} is devoted to the definition of the overall Lagrangian \(L\) and to the properties of the corresponding action functional; to avoid weighing down this \zcref[noref,nocap]{sec:minact}, we have decided to postpone the more technically and lengthy proofs to \zcref{sec:mincurve}. The Hamilton--Jacobi equation is introduced in \zcref{sec:problem}. \zcref[S]{sec:local} gathers a comparison principle and a representation formula for solutions to the local equations~\zcref{eq:introprob}. Following~\cite{Siconolfi22}, in \zcref{sec:disc} we define a semidiscrete problem and exploit its connection with our time-dependent equation to prove a comparison result. The representation formula is the main subject of \zcref{sec:solution}. Finally, \zcref{sec:Honbounded} records some results about one-dimensional Hamiltonians on an interval which are needed in this paper.

    \paragraph{Acknowledgments.} The author is a member of the INdAM research group GNAMPA.

    \section{Preliminaries}\label{sec:preliminaries}

    \subsection{Networks}

    We fix a dimension \(N\) and \(\Rds^{N}\) as ambient space. An \emph{embedded network}, or \emph{continuous graph}, is a subset \(\Gamma \subset \Rds^{N}\) of the form
    \begin{equation*}
        \Gamma = \bigcup_{\gamma \in \Ecal} \gamma([0, \lvert \gamma \rvert]) \subset \Rds^{N},
    \end{equation*}
    where \(\Ecal\) is a finite or countable collection of regular (i.e., \(\Crm^{1}\) with non-vanishing derivative) simple oriented curves, called \emph{arcs} of the network, with Euclidean length \(\lvert \gamma \rvert\) and parameterized by arc length in \([0, \lvert \gamma \rvert]\) (i.e., \(\lvert \dot{\gamma} \rvert \equiv 1\) for any \(\gamma \in \Ecal\)). Note that we are also assuming existence of one-sided derivatives at the endpoints \(0\) and \(\lvert \gamma \rvert\). We stress that a regular change of parameters does not affect our results.

    Observe that on the support of any arc \(\gamma\) we also consider the inverse parametrization defined as
    \begin{equation*}
        \wtgamma(s) \coloneqq \gamma(\lvert \gamma \rvert - s) \qquad \text{for \(s \in [0, \lvert \gamma \rvert]\)}.
    \end{equation*}
    We call \(\wtgamma\) the \emph{inverse arc} of \(\gamma\). We assume
    \begin{equation}\label{eq:nosovrap}
        \gamma((0, \lvert \gamma \rvert)) \cap \gamma'([0, \lvert \gamma \rvert]) = \emptyset \qquad \text{whenever \(\gamma' \ne \gamma\), \(\wtgamma\)}.
    \end{equation}

    We call \emph{vertices} the initial and terminal points of the arcs, and denote by \(\Vbf\) the sets of all such vertices. It follows from~\zcref{eq:nosovrap} that vertices are the only points where arcs with different support intersect and, in particular,
    \begin{equation*}
        \gamma((0, \lvert \gamma \rvert)) \cap \Vbf = \emptyset \qquad \text{for any \(\gamma \in \Ecal\)}.
    \end{equation*}

    We assume that the network is connected, namely given two vertices there is a finite concatenation of arcs linking them.

    For each \(x \in \Vbf\), we define \(\Ecal_{x} \coloneqq \{\gamma \in \Ecal : \gamma(0) = x\}\).

    We require that \(\Gamma\) satisfies the next conditions:
    \begin{enumerate}[label=\textbf{(\ensuremath{\mathbf{\Gamma}}\arabic*)},ref=\textnormal{\textbf{(\ensuremath{\mathbf{\Gamma}}\arabic*)}},font=\normalfont]\zcsetup{reftype=condenum}
        \item\label{condlengthbound} \(\inf\limits_{\gamma \in \Ecal} \lvert \gamma \rvert > 0\);
        \item \(\Gamma\) is locally finite, i.e., \(\Ecal_{x}\) is finite for any \(x \in \Vbf\).
    \end{enumerate}

    Usually, in the known literature about time-dependent Hamilton--Jacobi equations on networks, see, e.g., \cite{ImbertMonneau17,Siconolfi22,LionsSouganidis17}, \emph{loops}---any arc \(\gamma\) with \(\gamma(0) = \gamma(\lvert \gamma \rvert)\)---are not admitted. In this paper we do not assume this condition.

    The network \(\Gamma\) inherits a geodesic distance, denoted with \(d_{\Gamma}\), from the Euclidean metric of \(\Rds^{N}\). \zcref[S]{condlengthbound} ensures that \(d_{\Gamma}\) is a complete metric. It is clear that given \(x, y \in \Gamma\) there is at least a geodesic linking them.

    We consider the following differential structure on \(\Gamma\):
    \begin{mylist}
        \item If \(x \in \Gamma \setminus \Vbf\) let \(\gamma \in \Ecal\) and \(s \in (0, \lvert \gamma \rvert)\) be such that \(x = \gamma(s)\), then we set the \emph{tangent space} of \(\Gamma\) at \(x\) as
        \begin{equation*}
            \Trm_{x} \Gamma \coloneqq \left\{ q \in \Rds^{N} : \text{\(q = \lambda \dot{\gamma}(s)\) with \(\lambda \in \Rds\)} \right\}.
        \end{equation*}
        We point out that, according to~\zcref{eq:nosovrap}, \(\dot{\gamma}(s)\) is univocally determined up to a sign.
        \item Given \(x \in \Vbf\), we define \(\Trm_{x} \Gamma\) as the convex hull of the \(\lambda \dot{\gamma}(0)\) with \(\lambda \in \Rds\) and \(\gamma \in \Ecal_{x}\), i.e.,
        \begin{equation*}
            \Trm_{x} \Gamma \coloneqq \co \left\{ q \in \Rds^{N} : \text{\(q = \lambda \dot{\gamma}(0)\) with \(\lambda \in \Rds\) and \(\gamma \in \Ecal_{x}\)} \right\}.
        \end{equation*}
    \end{mylist}
    We further define the \emph{cotangent space} \(\Trm^{*}_{x} \Gamma\) as the dual space of \(\Trm_{x} \Gamma\), and set the \emph{tangent bundle} \(\Trm \Gamma\) and \emph{cotangent bundle} \(\Trm^{*} \Gamma\) as the disjoint union of the \(\Trm_{x} \Gamma\) and \(\Trm_{x}^{*} \Gamma\) at all points of \(\Gamma\), respectively.

    \subsection{Hamiltonians on \texorpdfstring{\(\Gamma\)}{Γ}}

    A \emph{Hamiltonian} on \(\Gamma\) is a collection of Hamiltonians \(\Hcal \coloneqq \{H_{\gamma}\}_{\gamma \in \Ecal}\), where
    \begin{align*}
        H_{\gamma} \colon [0, \lvert \gamma \rvert] \times \Rds^{+} \times \Rds & \longrightarrow \Rds \\
        (s, t, \mu) & \longmapsto H_{\gamma}(s, t, \mu),
    \end{align*}
    satisfying
    \begin{equation}\label{eq:Hcomp}
        H_{\wtgamma}(s, t, \mu) = H_{\gamma}(\lvert \gamma \rvert - s, t, - \mu) \qquad \text{for any arc \(\gamma\), \(s \in [0, \lvert \gamma \rvert]\), \(\mu \in \Rds\) and \(t \in \Rds^{+}\)}.
    \end{equation}
    We emphasize that,
    \begin{mylist}
        \item apart the above compatibility condition, the Hamiltonians \(H_{\gamma}\) are \emph{unrelated};
        \item we are not assuming any periodicity on \(H_{\gamma}\) when \(\gamma\) is a loop.
    \end{mylist}

    We require that
    \begin{Hassum}
        \item\label{condHcont} each \(H_{\gamma}\) is continuous;
        \item\label{condHconv} each \(H_{\gamma}\) is convex in \(\mu\);
        \item\label{condHcoer} for every \(T \in \Rds^{+}\) there exist the nondecreasing, convex and superlinearly coercive functions \(\vartheta_{T}, \Theta_{T} \colon \Rds^{+} \to \Rds\) such that, for any \(\gamma \in \Ecal\), \(s \in [0, \lvert \gamma \rvert]\), \(\mu \in \Rds\) and \(t \in [0, T]\):
        \begin{equation*}
            \vartheta_{T}(\lvert \mu \rvert) \le H_{\gamma}(s, t, \mu) \le \Theta_{T}(\lvert \mu \rvert);
        \end{equation*}
        \item\label{condHtlip} for each \(T \in \Rds^{+}\) there are \(\alpha_{T}, \beta_{T} \in \Rds^{+}\) and \(\epsilon_{T} \in (0, T]\) such that, for any \(\gamma \in \Ecal\), \(s \in [0, \lvert \gamma \rvert]\) and \(\mu \in \Rds\),
        \begin{equation*}
            H_{\gamma}(s, t_{1}, \mu) - H_{\gamma}(s, t_{2}, \mu) \le \left( \alpha_{T} \frac{H_{\gamma}(s, t_{2}, \mu)}{1 + \lvert \mu \rvert} + \beta_{T} \right) \lvert t_{2} - t_{1} \rvert,
        \end{equation*}
        whenever \(t_{1}, t_{2} \in [0, T]\) and \(\lvert t_{2} - t_{1} \rvert < \epsilon_{T}\).
    \end{Hassum}

    Under \zcref{condHcoer,condHconv} is natural to define, for any \(\gamma \in \Ecal\), the \emph{Lagrangian} corresponding to \(H_{\gamma}\) as
    \begin{equation*}
        L_{\gamma}(s, \lambda, t) \coloneqq \max_{\mu \in \Rds} \{\lambda \mu - H_{\gamma}(s, t, \mu)\}.
    \end{equation*}

    \begin{prop}\label{loclagprop}
        The Lagrangian \(L_{\gamma} \colon [0, 1] \times \Rds \times \Rds^{+} \to \Rds\) is a continuous map convex in the speed variable so that, for any arc \(\gamma\), \(s \in [0, \lvert \gamma \rvert]\), \(\lambda \in \Rds\) and \(t \in \Rds^{+}\),
        \begin{equation}\label{eq:loclagprop.1}
            L_{\wtgamma}(s, \lambda, t) = L_{\gamma}(\lvert \gamma \rvert - s, - \lambda, t).
        \end{equation}
        Furthermore, for every \(T > 0\),
        \begin{myenum}
            \item\label{en:loclagprop1} there exist two nondecreasing, convex and superlinearly continuous maps \(\whvartheta_{T}, \whTheta_{T} \colon \Rds^{+} \to \Rds\) such that, for any \(\gamma \in \Ecal\), \(s \in [0, \lvert \gamma \rvert]\), \(\lambda \in \Rds\) and \(t \in [0, T]\),
            \begin{equation*}
                \whvartheta_{T}(\lvert \lambda \rvert) \le L_{\gamma}(s, \lambda, t) \le \whTheta_{T}(\lvert \lambda \rvert);
            \end{equation*}
            \item\label{en:loclagprop2} there are \(\whalpha_{T}, \whbeta_{T} \in \Rds^{+}\) and \(\whepsilon_{T} \in (0, T]\) such that, for any \(\gamma \in \Ecal\),
            \begin{equation}\label{eq:loclagprop.2}
                L_{\gamma}(s, \lambda, t_{1}) - L_{\gamma}(s, \lambda, t_{2}) \le \left( \whalpha_{T} \lvert \lambda \rvert + \whbeta_{T} \right) \lvert t_{2} - t_{1} \rvert,
            \end{equation}
            whenever \(t_{1}, t_{2} \in [0, T]\), \(\lvert t_{2} - t_{1} \rvert < \whepsilon_{T}\), \(s \in [0, \lvert \gamma \rvert]\) and \(\mu \in \Rds\).
        \end{myenum}
    \end{prop}
    \begin{proof}
        Continuity and convexity are apparent, while~\zcref{eq:loclagprop.1} is a consequence of~\zcref{eq:Hcomp}. \zcref[S]{en:loclagprop1} can be deduced from \zcref{condHcoer} by setting \(\whvartheta_{T}\) and \(\whTheta_{T}\) as the convex conjugates of \(\Theta_{T}\) and \(\vartheta_{T}\), respectively. \\
        Next we prove \zcref{en:loclagprop2}. First we fix \(T > 0\) and let \(\epsilon_{T}\) be as in \zcref{condHtlip}, then, given \(\gamma \in \Ecal\), \(s \in [0, \lvert \gamma \rvert]\), \(\lambda \in \Rds\) and \(t_{1}, t_{2} \in [0, T]\) with \(\lvert t_{1} - t_{2} \rvert < \epsilon_{T}\), we choose a \(\mu \in \Rds\) such that
        \begin{equation*}
            L_{\gamma}(s, \lambda, t_{1}) = \lambda \mu - H_{\gamma}(s, t_{1}, \mu).
        \end{equation*}
        Exploiting \zcref{condHtlip} we then have
        \begin{equation}\label{eq:loclagprop1}
            \begin{aligned}
                L_{\gamma}(s, \lambda, t_{1}) - L_{\gamma}(s, \lambda, t_{2}) & \le \lambda \mu - H_{\gamma}(s, t_{1}, \mu) - \lambda \mu + H_{\gamma}(s, t_{2}, \mu) \\
                & \le \left( \alpha_{T} \frac{H_{\gamma}(s, t_{1}, \mu)}{1 + \lvert \mu \rvert} + \beta_{T} \right) \lvert t_{2} - t_{1} \rvert.
            \end{aligned}
        \end{equation}
        To conclude we observe that, according to \zcref{en:loclagprop1},
        \begin{equation*}
            - \left\lvert \whvartheta_{T}(0) \right\rvert + H_{\gamma}(s, t_{1}, \mu) \le L_{\gamma}(s, \lambda, t_{1}) + H_{\gamma}(s, t_{1}, \mu) = \lambda \mu \le \lvert \lambda \rvert \lvert \mu \rvert < \lvert \lambda \rvert(1 + \lvert \mu \rvert),
        \end{equation*}
        thereby
        \begin{equation*}
            \frac{H_{\gamma}(s, t_{1}, \mu)}{1 + \lvert \mu \rvert} < \lvert \lambda \rvert + \left\lvert \whvartheta_{T}(0) \right\rvert,
        \end{equation*}
        which together with~\zcref{eq:loclagprop1} proves~\zcref{eq:loclagprop.2}.
    \end{proof}

    \section{Minimal Action Functional}\label{sec:minact}

    \subsection{The Lagrangian on \texorpdfstring{\(\Gamma\)}{Γ}}

    We define an overall Lagrangian \(L\) on the tangent bundle of \(\Gamma\) by gluing together the local Lagrangians \(L_{\gamma}\). Central to this construction is the flux limiter, a function defined on \(\Vbf\) that characterizes the behavior of \(L\) on the vertices.

    \begin{defin}\label{fldef}
        A \emph{flux limiter} is any map \((x, t) \mapsto c_{x}(t)\) from \(\Vbf \times \Rds^{+}\) to \(\Rds\) continuous in \(t\) and such that
        \begin{equation}\label{eq:fldef}
            c_{x}(t) \le \min_{\gamma \in \Ecal_{x}} L_{\gamma}(0, 0, t) = \min_{\gamma \in \Ecal_{x}} \max_{\mu \in \Rds} - H_{\gamma}(0, t, \mu) \qquad \text{for any \((x, t) \in \Vbf \times \Rds^{+}\)}.
        \end{equation}
    \end{defin}

    In this paper we will only consider flux limiters satisfying the following two conditions:
    \begin{enumerate}[label=\textbf{(FL\arabic*)},ref=\textnormal{\textbf{(FL\arabic*)}},font=\normalfont]\zcsetup{reftype=condenum}
        \item\label{condflbound} for each \(T \in \Rds^{+}\) there exists a constant \(C_{T}\) such that
        \begin{equation*}
            \lvert c_{x}(t) \rvert \le C_{T} \qquad \text{for any \(x \in \Gamma\) and \(t \in [0, T]\)};
        \end{equation*}
        \item\label{condfllip} for each \(T \in \Rds^{+}\) there are \(\whbeta_{T} \in \Rds^{+}\) and \(\whepsilon_{T} \in (0, T]\) such that, for any \(x \in \Vbf\),
        \begin{equation*}
            c_{x}(t_{1}) - c_{x}(t_{2}) \le \whbeta_{T} \lvert t_{2} - t_{1} \rvert
        \end{equation*}
        whenever \(t_{1}, t_{2} \in [0, T]\) and \(\lvert t_{2} - t_{1} \rvert < \whepsilon_{T}\).
    \end{enumerate}

    The next \zcref[noref]{maxfl}, which ensures that these conditions hold true for the maximal flux limiter, is an easy consequence of \zcref{loclagprop}.

    \begin{prop}\label{maxfl}
        The flux limiter
        \begin{equation*}
            (x, t) \longmapsto \min_{\gamma \in \Ecal_{x}} L_{\gamma}(0, 0, t)
        \end{equation*}
        satisfies \zcref{condfllip,condflbound}.
    \end{prop}

    Fixed a flux limiter \(c_{x}\), we define an overall Lagrangian \(L \colon \Trm \Gamma \times \Rds^{+} \to \Rds\) in the following way:
    \begin{mylist}
        \item if \(x \notin \Vbf\), i.e., \(x = \gamma(s)\) for some \(\gamma \in \Ecal\) and \(s \in (0, \lvert \gamma \rvert)\),
        \begin{equation*}
            L(x, q, t) \coloneqq L_{\gamma}(s, q \cdot \dot{\gamma}(s), t) ;
        \end{equation*}
        \item if \(x \in \Vbf\) we set
        \begin{equation*}
            L(x, q, t) \coloneqq \frac{\lvert q \rvert^{2}}{2} + c_{x}(t).
        \end{equation*}
    \end{mylist}

    Exploiting~\zcref{loclagprop,condfllip,condflbound} we get:

    \begin{prop}\label{gLagprop}
        The Lagrangian \(L\) is a well-defined measurable function on \(\Trm \Gamma \times \Rds^{+}\), is convex in the speed variable and for each \(T \in \Rds^{+}\)
        \begin{mylist}
            \item there exist the nondecreasing, convex and superlinearly coercive functions \(\ovvartheta_{T}, \ovTheta_{T} \colon \Rds^{+} \to \Rds\) such that, for any \((x, q) \in \Trm \Gamma\) and \(t \in [0, T]\):
            \begin{equation*}
                \ovvartheta_{T}(\lvert q \rvert) \le L(x, q, t) \le \ovTheta_{T}(\lvert q \rvert) ;
            \end{equation*}
            \item there are \(\ovalpha_{T}, \ovbeta_{T} \in \Rds^{+}\) and \(\ovepsilon_{T} \in (0, T]\) such that
            \begin{equation*}
                L(x, q, t_{1}) - L(x, q, t_{2}) \le \left( \ovalpha_{T} \lvert q \rvert + \ovbeta_{T} \right) \lvert t_{2} - t_{1} \rvert,
            \end{equation*}
            whenever \((x, q) \in \Trm \Gamma\), \(t_{1}, t_{2} \in [0, T]\) and \(\lvert t_{2} - t_{1} \rvert < \ovepsilon_{T}\).
        \end{mylist}
    \end{prop}

    Given a curve \(\xi \colon [t_{1}, t_{2}] \to \Gamma\), we define the corresponding action functional as
    \begin{equation*}
        \int_{t_{1}}^{t_{2}} L \left( \xi, \dot{\xi}, \tau \right) \dl \tau.
    \end{equation*}

    \begin{theo}\label{actlsc}
        Let \(\{\xi_{n}\}_{n \in \Nds}\) be a sequence of curves defined in \([t_{1}, t_{2}]\), uniformly converging to a curve \(\zeta \colon [t_{1}, t_{2}] \to \Gamma\) and such that \(\left\{ \dot{\xi}_{n} \right\}\) weakly converges to \(\dot{\zeta}\) in \(\Lrm^{2}\). Then
        \begin{equation}\label{eq:actlsc.1}
            \liminf_{n \to \infty} \int_{t_{1}}^{t_{2}} L \left( \xi_{n}, \dot{\xi}_{n}, \tau \right) \dl \tau \ge \int_{t_{1}}^{t_{2}} L \left( \zeta, \dot{\zeta}, \tau \right) \dl \tau.
        \end{equation}
    \end{theo}

    To prove this semicontinuity result we need a preliminary \zcref[noref]{WlscL}.

    \begin{lem}\label{WlscL}
        Let \(\{\xi_{n}\}_{n \in \Nds}\) be a sequence of curves defined in \([t_{1}, t_{2}]\), uniformly converging to a vertex \(z\) and such that \(\left\{ \dot{\xi}_{n} \right\}\) weakly converges to \(0\) in \(\Lrm^{2}\). Then
        \begin{equation}\label{eq:WlscL.1}
            \liminf_{n \to \infty} \int_{t_{1}}^{t_{2}} L \left( \xi_{n}, \dot{\xi}_{n}, \tau \right) \dl \tau \ge \int_{t_{1}}^{t_{2}} c_{z}(\tau) \dl \tau.
        \end{equation}
    \end{lem}
    \begin{proof}
        First we extend \(L\) to \(\Rds^{N} \times \Rds^{N} \times \Rds^{+}\) by assuming that \(L(x, q, t) = \infty\) whenever \((x, q) \notin \Trm \Gamma\), then we define \(L^{*} \colon \Rds^{N} \times \Rds^{N} \times \Rds^{+} \to \Rds\) as
        \begin{equation*}
            L^{*}(x, p, t) \coloneqq \sup_{q \in \Rds^{N}} \{pq - L(x, q, t)\} \qquad \text{for \((x, p, t) \in \Rds^{N} \times \Rds^{N} \times \Rds^{+}\)}.
        \end{equation*}
        We know from the theory of convex conjugate functions that
        \begin{equation}\label{eq:WlscL1}
            L(x, q, t) = \sup_{p \in \Rds^{N}} \{pq - L^{*}(x, p, t)\} \qquad \text{for all \((x, q, t) \in \Rds^{N} \times \Rds^{N} \times \Rds^{+}\)}.
        \end{equation}
        It is also apparent that whenever \(x = \gamma(s)\) for some \(\gamma \in \Ecal\) and \(s \in (0, \lvert \gamma \rvert)\)
        \begin{equation}\label{eq:WlscL2}
            L^{*}(x, p, t) = H_{\gamma}(s, t, p \cdot \dot{\gamma}(s)),
        \end{equation}
        while for every \(x \in \Vbf\)
        \begin{equation}\label{eq:WlscL3}
            L^{*}(x, p, t) = \frac{\lvert p \rvert^{2}}{2} - c_{x}(t).
        \end{equation}
        Next we denote with \(\{\gamma_{1}, \dotsc, \gamma_{m}\}\) the collection of arcs which make up the set \(\Ecal_{z}\), define for each \(i \in \{1, \dotsc, m\}\) the maps \(p_{i} \in \Lrm^{\infty}([t_{1}, t_{2}])\) such that
        \begin{equation}\label{eq:WlscL4}
            H_{\gamma_{i}}(0, t, p_{i}(t)) = - L_{\gamma_{i}}(0, 0, t) \qquad \text{for a.e.\ \(t \in [t_{1}, t_{2}]\)},
        \end{equation}
        which exist by \zcref{selection}, and set the \(\Lrm^{\infty}\) function
        \begin{equation*}
            p(x, t) \coloneqq
            \begin{dcases}
                p_{i}(t) \dot{\gamma}_{i}(0) & \text{if \(x \in \gamma_{i}((0, \lvert \gamma_{i} \rvert))\)}, \\
                0 & \text{otherwise},
            \end{dcases}
            \qquad \text{for \((x, t) \in \Rds^{N} \times [t_{1}, t_{2}]\)}.
        \end{equation*}
        According to~\zcref{eq:WlscL2,eq:WlscL3} we have that either
        \begin{equation*}
            L^{*}(\xi_{n}(t), p(\xi_{n}(t), t), t) = - c_{z}
        \end{equation*}
        or
        \begin{equation*}
            L^{*}(\xi_{n}(t), p(\xi_{n}(t), t), t) = H_{\gamma_{i}} \left( \gamma^{- 1}_{i}(\xi_n(t)), t, p_{i}(t) \right) \qquad \text{for some \(i \in \{1, \dotsc, m\}\)},
        \end{equation*}
        consequently the uniform convergence of \(\{\xi_{n}\}\), \zcref{eq:WlscL4,fldef} imply
        \begin{equation}\label{eq:WlscL5}
            \limsup_{n \to \infty} L^{*}(\xi_{n}(t), p(\xi_{n}(t), t), t) \le - c_{z} \qquad \text{for a.e.\ \(t \in [t_{1}, t_{2}]\)}.
        \end{equation}
        We finally get from~\zcref{eq:WlscL1}
        \begin{align*}
            \liminf_{n \to \infty} \int_{t_{1}}^{t_{2}} L \left( \xi_{n}, \dot{\xi}_{n}, \tau \right) \dl \tau & = \liminf_{n \to \infty} \int_{t_{1}}^{t_{2}} \sup_{p \in \Rds} \left( p \cdot \dot{\xi}_{n}(\tau) - L^{*}(\xi_{n}, p, \tau) \right) \dl \tau \\
            & \ge \liminf_{n \to \infty} \int_{t_{1}}^{t_{2}} \left( p(\xi_{n}(\tau), \tau) \cdot \dot{\xi}_{n}(\tau) - L^{*}(\xi_{n}(\tau), p(\xi_{n}(\tau), \tau), \tau) \right) \dl \tau,
        \end{align*}
        then the weak convergence of \(\left\{ \dot{\xi}_{n} \right\}\), \zcref{eq:WlscL5} and Fatou's Lemma yield~\zcref{eq:WlscL.1}.
    \end{proof}

    \begin{proof}[Proof of \zcref{actlsc}]
        Let \(E_{1}\) be the set made up by the union of the proper subintervals \(I\) of \([t_{1}, t_{2}]\) such that \(\zeta(I) \subseteq \Vbf\). It is apparent that \(E_{1}\), if not empty, is the union of an at most countable collection of closed intervals and
        \begin{equation*}
            \dot{\zeta}(t) = 0 \qquad \text{for a.e.\ \(t \in E_{1}\)},
        \end{equation*}
        thus \zcref{WlscL} yields
        \begin{equation}\label{eq:actlsc1}
            \liminf_{n \to \infty} \int_{E_{1}} L \left( \xi_{n}, \dot{\xi}_{n}, \tau \right) \dl \tau \ge \int_{E_{1}} L \left( \zeta, \dot{\zeta}, \tau \right) \dl \tau.
        \end{equation}
        Let us now denote by \(E_{2}\) the open set \((t_{1}, t_{2}) \setminus E_{1}\). According to~\zcref{eq:actlsc1}, to prove~\zcref{eq:actlsc.1} it is enough to show that
        \begin{equation}\label{eq:actlsc2}
            \liminf_{n \to \infty} \int_{E_{2}} L \left( \xi_{n}, \dot{\xi}_{n}, \tau \right) \dl \tau \ge \int_{E_{2}} L \left( \zeta, \dot{\zeta}, \tau \right) \dl \tau.
        \end{equation}
        We start by taking the lower semicontinuous envelope
        \begin{equation*}
            \whL(x, q, t) \coloneqq \lim_{r \to 0^{+}} \inf \left\{ L \left( x', q', t' \right) : \left( x', q', t' \right) \in \left( \Trm \Gamma \times \Rds^{+} \right) \cap B_{r}((x, q, t)) \right\}
        \end{equation*}
        and setting \(\ovL \colon \Trm \Gamma \times \Rds^{+} \to \Rds\) as the largest convex function in \(q\) bounded above by \(\whL\), see~\cite[Corollary~4.22]{Clarke13}. We clearly have that
        \begin{equation*}
            L(x, q, t) \ge \whL(x, q, t) \ge \ovL(x, q, t) \qquad \text{for every \((x, q, t) \in \Trm \Gamma \times \Rds^{+}\)},
        \end{equation*}
        which becomes an identity whenever \(x \notin \Vbf\), thereby
        \begin{equation}\label{eq:actlsc3}
            L \left( \xi_{n}, \dot{\xi}_{n}, t \right) \ge \ovL \left( \xi_{n}, \dot{\xi}_{n}, t \right) \qquad \text{for all \(t \in E_{2}\)},
        \end{equation}
        and, since \(\zeta(t) \notin \Vbf\) for a.e.\ \(t \in E_{2}\),
        \begin{equation}\label{eq:actlsc4}
            L \left( \zeta, \dot{\zeta}, t \right) = \ovL \left( \zeta, \dot{\zeta}, t \right) \qquad \text{for a.e.\ \(t \in E_{2}\)}.
        \end{equation}
        It is straightforward to check that \(\ovL\) satisfies the conditions of a classic lower semicontinuity result for integral functionals (see, e.g., \cite[Theorem~6.38]{Clarke13}), therefore
        \begin{equation*}
            \liminf_{n \to \infty} \int_{E_{2}} \ovL \left( \xi_{n}, \dot{\xi}_{n}, \tau \right) \dl \tau \ge \int_{E_{2}} \ovL \left( \zeta, \dot{\zeta}, \tau \right) \dl \tau,
        \end{equation*}
        which together with~\zcref{eq:actlsc4,eq:actlsc3} proves~\zcref{eq:actlsc2}.
    \end{proof}

    \subsection{Optimal Curves}

    We consider the set
    \begin{equation*}
        \Dcal \coloneqq \left\{ (x_{1}, t_{1}, x_{2}, t_{2}) \in \left( \Gamma \times \Rds^{+} \right)^{2} : \text{\(t_{2} > t_{1}\) or \(t_{1} = t_{2}\), \(x_{1} = x_{2}\)} \right\}
    \end{equation*}
    and the \emph{minimal action functional}
    \begin{equation}\label{eq:minactdef}
        \Phi(x_{1}, t_{1}, x_{2}, t_{2}) \coloneqq \inf \left\{ \int_{t_{1}}^{t_{2}} L \left( \xi, \dot{\xi}, \tau \right) \dl \tau \right\} \qquad \text{for \((x_{1}, t_{1}, x_{2}, t_{2}) \in \Dcal\)},
    \end{equation}
    where the infimum is taken over the curves \(\xi \colon [t_{1}, t_{2}] \to \Gamma\) with \(\xi(t_{1}) = x_{1}\) and \(\xi(t_{2}) = x_{2}\). We further assume that
    \begin{equation*}
        \Phi(x_{1}, t_{1}, x_{2}, t_{2}) \coloneqq \infty \qquad \text{whenever \(t_{1} = t_{2}\) and \(x_{1} \ne x_{2}\).}
    \end{equation*}
    Note that for any pair of points \(x_{1}\) and \(x_{2}\) in \(\Gamma\) and \(t_{2} > t_{1} \ge 0\) there are curves defined in \([t_{1}, t_{2}]\) linking \(x_{1}\) to \(x_{2}\) with finite action functional. It is enough to take a geodesic linking \(x_{1}\) to \(x_{2}\), which does exist since the network is connected, and to change the parametrization in order to define it in \([t_{1}, t_{2}]\). Moreover, we can always select an optimal curve which realizes the infimum in~\zcref{eq:minactdef}.

    \begin{restatable}{theo}{mincurvelip}\label{mincurvelip}
        Fixed \(T > 0\) and \(C > 0\), there is a positive constant \(\kappa\) such that, for every \(x_{1}, x_{2} \in \Gamma\) and \(t_{1}, t_{2} \in [0, T]\) satisfying
        \begin{equation*}
            d_{\Gamma}(x_{1}, x_{2}) \le C(t_{2} - t_{1}),
        \end{equation*}
        there exists an optimal \(\kappa\)-Lipschitz continuous curve for \(\Phi(x_{1}, t_{1}, x_{2}, t_{2})\).
    \end{restatable}

    The Lipschitz continuity of the optimal curves yields a Lipschitz continuity result for the minimal action.

    \begin{restatable}{cor}{minactlip}\label{minactlip}
        Given \(T > 0\) and \(C > 0\), the minimal action \(\Phi\) is Lipschitz continuous on
        \begin{equation*}
            A_{C} \coloneqq \left\{ (x_{1}, t_{1}, x_{2}, t_{2}) \in (\Gamma \times [0, T])^{2} : d_{\Gamma}(x_{1}, x_{2}) \le C(t_{2} - t_{1}) \right\}.
        \end{equation*}
    \end{restatable}

    The proof of these results are in \zcref{sec:minact}.

    \subsection{Zeno Behavior}\label{sec:nozeno}

    Under our assumptions, an optimal curve for the minimal action functional could exhibit the so-called \emph{Zeno behavior}, namely a minimizer could oscillate infinitely many times around one or more vertices. Notoriously, this phenomenon can be disruptive to both analysis and simulations. Although our study is not affected by such behavior, we provide here a criterion to rule it out, motivated by potential future applications.

    Our condition to avoid Zeno behavior is rather technical; however, it is implied by a simpler and fairly general assumption, namely that the chosen flux limiter is strictly smaller than the maximal one:
    \begin{equation}\label{eq:suff_nozeno}
        c_{x}(t) < \min_{\gamma \in \Ecal_{x}} L_{\gamma}(0, 0, t) \qquad \text{for all \(x \in \Vbf\) and \(t \in \Rds^{+}\)}.
    \end{equation}

    To prove our criterion we need to introduce a new class of curves.

    \begin{defin}
        We call a curve \(\xi \colon [t_{1}, t_{2}] \to \Gamma\) \emph{admissible} if there are at most finitely many intervals \([a, b] \subseteq [t_{1}, t_{2}]\) such that \(\xi(a) = \xi(b) \in \Vbf\) and \(\xi((a, b)) \cap \Vbf = \emptyset\).
    \end{defin}

    It is apparent that a curve is admissible if and only if it does not exhibit Zeno behavior.

    \begin{prop}\label{adcurve}
        Given \(T > 0\) and a \(\kappa\)-Lipschitz continuous curve \(\xi \colon [t_{1}, t_{2}] \subseteq [0, T] \to \Gamma\), assume that the flux limiter \(c_{x}\) satisfies the following inequality for some \(\varepsilon > 0\):
        \begin{equation}\label{eq:adcurve.1}
            c_{x}(t) \le \min_{\substack{\gamma \in \Ecal_{x}, \, s \in [0, \frac{\kappa \varepsilon}{2}], \\ r \in [t - \varepsilon, t + \varepsilon] \cap [0, T]}} L_{\gamma}(s, 0, r) \qquad \text{for any \(t \in [t_{1}, t_{2}]\) and \(x \in \xi([t_{1}, t_{2}]) \cap \Vbf\)}.
        \end{equation}
        Then, there is an admissible \(\kappa\)-Lipschitz continuous curve \(\zeta \colon [t_{1}, t_{2}] \to \Gamma\) with the same endpoints of \(\xi\) such that
        \begin{equation*}
            \int_{t_{1}}^{t_{2}} L \left( \xi, \dot{\xi}, \tau \right) \dl \tau \ge \int_{t_{1}}^{t_{2}} L \left( \zeta, \dot{\zeta}, \tau \right) \dl \tau.
        \end{equation*}
    \end{prop}
    \begin{proof}
        If \(\xi\) is admissible there is nothing to prove, thus we assume that it is not the case, i.e., that there is an infinite collection of nondegenerate intervals \([a_{n}, b_{n}] \subset [t_{1}, t_{2}]\) with disjoint interior such that \(\xi(a_{n}) = \xi(b_{n}) \in \Vbf\) and \(\xi((a_{n}, b_{n})) \cap \Vbf = \emptyset\). Note that they are countable many because \(\Rds\) is separable. Choosing an interval \([a_{k}, b_{k}]\) with \(b_{k} - a_{k} \le \varepsilon\), there is a vertex \(x = \xi(a_{k}) = \xi(b_{k})\) and a \(\gamma \in \Ecal_{x}\) such that \(\xi([a_{k}, b_{k}])\) is contained in the support of the arc \(\gamma\), while according to~\cite[Lemma~3.2]{PozzaSiconolfi23} \(\eta \coloneqq \gamma^{- 1} \circ \xi |_{[a_{k}, b_{k}]}\) is \(\kappa\)-Lipschitz continuous. It follows from~\zcref{eq:adcurve.1} and \zcref{nozenoaux} that
        \begin{equation*}
            \int_{a_{k}}^{b_{k}} L \left( \xi, \dot{\xi}, \tau \right) \dl \tau = \int_{a_{k}}^{b_{k}} L_{\gamma}(\eta, \dot{\eta}, \tau) \dl \tau \ge \int_{a_{k}}^{b_{k}} c_{x}(\tau) \dl \tau,
        \end{equation*}
        which in turn shows that, if we define the \(\kappa\)-Lipschitz continuous curve
        \begin{equation*}
            \zeta(t) \coloneqq
            \begin{dcases}
                \xi(a_{n}) & \text{if \(t \in [a_{n}, b_{n}]\) with \(b_{n} - a_{n} \le \varepsilon\)}, \\
                \xi(t) & \text{otherwise},
            \end{dcases}
        \end{equation*}
        we then have
        \begin{equation*}
            \int_{t_{1}}^{t_{2}} L \left( \xi, \dot{\xi}, \tau \right) \dl \tau \ge \int_{t_{1}}^{t_{2}} L \left( \zeta, \dot{\zeta}, \tau \right) \dl \tau.
        \end{equation*}
        This concludes the proof because \(\zeta\) is admissible. Indeed, if \([a, b] \subseteq [t_{1}, t_{2}]\) is an interval such that \(\zeta(a) = \zeta(b) \in \Vbf\) and \(\zeta((a, b)) \cap \Vbf = \emptyset\), then \(b - a > \varepsilon\). There are clearly only finitely many such intervals, hence \(\zeta\) is admissible.
    \end{proof}

    \begin{rem}
        We deduce from \zcref{adcurve} the following sufficient condition to avoid the Zeno phenomenon: for each \(T > 0\) and \(x \in \Vbf\) there exist two positive constants \(\delta_{T, x}\) and \(\varepsilon_{T, x}\) such that
        \begin{equation}\label{eq:nozeno1}
            c_{x}(t) \le \min_{\substack{\gamma \in \Ecal_{x}, \, s \in [0, \delta_{T, x}], \\ r \in [t - \varepsilon_{T, x}, t + \varepsilon_{T, x}] \cap [0, T]}} L_{\gamma}(s, 0, r) \qquad \text{for any \(t \in [0, T]\)}.
        \end{equation}
        Indeed, if \(\xi \colon [t_{1}, t_{2}] \to \Gamma\) is a \(\kappa\)-Lipschitz continuous optimal curve for \(\Phi(x_{1}, t_{1}, x_{2}, t_{2})\), then the flux limiter \(c_{x}\) satisfies~\zcref{eq:adcurve.1} with
        \begin{equation*}
            \varepsilon = \min \left\{ \varepsilon_{T, x}, \frac{2 \delta_{T, x}}{\kappa} \right\},
        \end{equation*}
        where the minimum is taken over the \(x \in \xi([t_{1}, t_{2}]) \cap \Vbf\). This set consists of finitely many vertices due to \zcref{condlengthbound} and the Lipschitz continuity of \(\xi\). In view of \zcref{adcurve}, we can thus assume that \(\xi\) is also admissible. \\
        Let us review some cases in which~\zcref{eq:nozeno1} is satisfied.
        \begin{mylist}
            \item In the autonomous stationary case, namely when the Lagrangians are independent of time and state variables, \zcref{eq:nozeno1} holds true for any flux limiter.
            \item In the autonomous case analyzed in~\cite{PozzaSiconolfi23}, the flux limiters satisfy the condition
            \begin{equation*}
                c_{x} \le \min_{\gamma \in \Ecal_{x}} \min_{s \in [0, \lvert \gamma \rvert]} L_{\gamma}(s, 0),
            \end{equation*}
            which implies~\zcref{eq:nozeno1}.
            \item We now prove that~\zcref{eq:suff_nozeno} implies~\zcref{eq:nozeno1}. For each \(x \in \Vbf\) and \(T > 0\), the map \((s, t) \mapsto \min\limits_{\gamma \in \Ecal_{x}} L_{\gamma}(s, 0, t)\) admits a modulus of continuity \(\omega_{T, x}\) in \(\left[ 0, \min\limits_{\gamma \in \Ecal_{x}} \lvert \gamma \rvert \right] \times [0, T]\). If \(c_{x}\) satisfies~\zcref{eq:suff_nozeno} we can then choose \(\varepsilon_{T, x}\) and \(\delta_{T, x}\) such that, for any \(t \in [0, T]\),
            \begin{equation*}
                c_{x}(t) \le \min_{\gamma \in \Ecal_{x}} L_{\gamma}(0, 0, t) - \omega_{T, x}(\varepsilon_{T, x} + \delta_{T, x}) \le \min_{\substack{\gamma \in \Ecal_{x}, \, s \in [0, \delta_{T, x}], \\ r \in [t - \varepsilon_{T, x}, t + \varepsilon_{T, x}] \cap [0, T]}} L_{\gamma}(s, 0, r).
            \end{equation*}
        \end{mylist}
    \end{rem}

    \section{Setting of the Problem}\label{sec:problem}

    We consider the time-dependent equation
    \begin{equation}\label{eq:globHJ}\tag{\ensuremath{\Hcal}J}
        \difcp{u(x, t)}{t} + \Hcal(x, t, \difcp{u}{x}) \qquad \text{in \(\Gamma \times (0, \infty)\)}.
    \end{equation}
    This notation synthetically indicates the family, for \(\gamma\) varying in \(\Ecal\), of Hamilton--Jacobi equations
    \begin{equation}\label{eq:locHJ}\tag{HJ\textsubscript{\ensuremath{\gamma}}}
        \difcp{U(s, t)}{t} + H_{\gamma}(s, t, \difcp{U}{s}) \qquad \text{in \((0, \lvert \gamma \rvert) \times (0, \infty)\)}.
    \end{equation}

    Here (sub/super)solutions to the local problem~\eqref{eq:locHJ} are intended in the viscosity sense, see for instance~\cite{BardiCapuzzo-Dolcetta97} for a comprehensive treatment of viscosity solutions theory. We just recall that, given an open set \(\Ocal\) and a continuous function \(u \colon \ovOcal \to \Rds\), a \emph{supertangent} (resp.\ \emph{subtangent}) to \(u\) at \(x \in \Ocal\) is a viscosity test function from above (resp.\ below). If needed, we take, without explicitly mentioning, \(u\) and test function coinciding at \(x\) and test function strictly greater (resp.\ less) than \(u\) in a punctured neighborhood of \(x\). We say that a subtangent \(\varphi\) to \(u\) at \(x \in \partial \Ocal\) is \emph{constrained to \(\ovOcal\)} if \(x\) is a minimizer of \(u - \varphi\) in a neighborhood of \(x\) intersected with \(\ovOcal\).

    We are interested in finding a continuous function \(u \colon \Gamma \times \Rds^{+} \to \Rds\) such that \(u(\gamma(s), t)\) is a viscosity solution to~\zcref{eq:locHJ} for any arc \(\gamma\), taking into account a given flux limiter. Our definition of (sub/super)solutions to~\zcref{eq:globHJ} is an extension of the one given in~\cite{Siconolfi22} for the autonomous case.

    \begin{defin}\label{defsol}
        We say that a continuous \(w \colon \Gamma \times \Rds^{+} \to \Rds\) is a \emph{subsolution} to~\zcref{eq:globHJ} with flux limiter \(c_{x}\) if:
        \begin{myenum}
            \item\label{en:subsol1} \((s, t) \mapsto w(\gamma(s), t)\) is a viscosity subsolution to~\zcref{eq:locHJ} for each \(\gamma \in \Ecal\);
            \item\label{en:subsol2} for any \(t_{0} \in (0, \infty)\) and vertex \(x\), if \(\psi\) is a \(\Crm^{1}\) supertangent to \(w(x, \cdot)\) at \(t_{0}\) then \(\difcp{\psi(t_{0})}{t} \le c_{x}(t_{0})\).
        \end{myenum}
        We say that a continuous \(v \colon \Gamma \times \Rds^{+} \to \Rds\) is a \emph{supersolution} to~\zcref{eq:globHJ} if:
        \begin{myenum}[resume]
            \item\label{en:supsol1} \((s, t) \mapsto v(\gamma(s), t)\) is a viscosity supersolution to~\zcref{eq:locHJ} for any \(\gamma \in \Ecal\);
            \item\label{en:supsol2} for every vertex \(x\) and \(t_{0} \in (0, \infty)\), if \(\phi\) is a \(\Crm^{1}\) subtangent to \(v(x, \cdot)\) at \(t_{0}\) such that \(\difcp{\phi(t_{0})}{t} < c_{x}\), then there is a \(\gamma \in \Ecal_{x}\) such that
            \begin{equation*}
                \difcp{\varphi(0, t_{0})}{t} + H_{\gamma}(0, t_{0}, \difcp{\varphi(0, t_{0})}{s}) \ge 0
            \end{equation*}
            for any \(\varphi\) that is a constrained \(\Crm^{1}\) subtangent to \((s, t) \mapsto v(\gamma(s), t)\) at \((0, t_{0})\). Notice that this condition does not require the existence of constrained subtangents.
        \end{myenum}
        We say that \(u \colon \Gamma \times \Rds^{+} \to \Rds\) is a \emph{solution} to~\zcref{eq:globHJ} if it is both a subsolution and supersolution.
    \end{defin}

    \section{Local Analysis on Arcs}\label{sec:local}

    We fix an arc \(\gamma \in \Ecal\) and focus on the local equation~\zcref{eq:locHJ}. Our treatment is independent of whether or not \(\gamma\) is a loop.

    First, to ease notation, we define the sets
    \begin{align*}
        \Qcal_{\gamma} & \coloneqq (0, \lvert \gamma \rvert) \times (0, \infty), & \Qcal_{\gamma, T} & \coloneqq (0, \lvert \gamma \rvert) \times (0, T),
        \intertext{and their partial boundaries}
        \partial_{p}^{-} \Qcal_{\gamma} & \coloneqq ([0, \lvert \gamma \rvert] \times \{0\}) \cup \left( \{0\} \times \Rds^{+} \right), & \partial_{p}^{-} \Qcal_{\gamma, T} & \coloneqq ([0, \lvert \gamma \rvert] \times \{0\}) \cup (\{0\} \times [0, T]), \\
        \partial_{p}^{+} \Qcal_{\gamma} & \coloneqq ([0, \lvert \gamma \rvert] \times \{0\}) \cup \left( \{\lvert \gamma \rvert\} \times \Rds^{+} \right), & \partial_{p}^{+} \Qcal_{\gamma, T} & \coloneqq ([0, \lvert \gamma \rvert] \times \{\lvert \gamma \rvert\}) \cup (\{\lvert \gamma \rvert\} \times [0, T]), \\
        \partial_{p} \Qcal_{\gamma} & \coloneqq \partial_{p}^{-} \Qcal_{\gamma} \cup \partial_{p}^{+} \Qcal_{\gamma}, & \partial_{p} \Qcal_{\gamma, T} & \coloneqq \partial_{p}^{-} \Qcal_{\gamma, T} \cup \partial_{p}^{+} \Qcal_{\gamma, T}.
    \end{align*}

    We prove a comparison result for~\zcref{eq:locHJ} based on the so-called doubling variable method. We rely on a variation of a classical result, see for instance~\cite[Lemma~3.1]{CrandallIshiiLions92}, which can be easily proved using the same arguments.

    \begin{lem}\label{alphamax}
        Given \(T > 0\), \(\delta \in (0, 1)\) and the continuous maps \(w, v \colon \ovQcal_{\gamma, T} \to \Rds\), we set, for \(\rho, \sigma \in \Rds^{+}\),
        \begin{equation*}
            M_{\rho, \sigma} \coloneqq \max_{(s, t), (s', t') \in \ovQcal_{\gamma, T}} \left\{ w(s, t) - v \left( s', t' \right) - \frac{1}{2} \left\lvert \rho \left( s - s' \right) - \delta \right\rvert^{2} - \frac{\sigma}{2} \left\lvert t - t' \right\rvert^{2} \right\}.
        \end{equation*}
        Assume that \(\lim\limits_{\rho, \sigma \to \infty} M_{\rho, \sigma} < \infty\) and let \((s_{\rho, \sigma}, t_{\rho, \sigma}), \left( s'_{\rho, \sigma}, t'_{\rho, \sigma} \right)\) be such that
        \begin{equation*}
            M_{\rho, \sigma} = w(s_{\rho, \sigma}, t_{\rho, \sigma}) - v \left( s'_{\rho, \sigma}, t'_{\rho, \sigma} \right) - \frac{1}{2} \left\lvert \rho \left( s_{\rho, \sigma} - s'_{\rho, \sigma} \right) - \delta \right\rvert^{2} - \frac{\sigma}{2} \left\lvert t_{\rho, \sigma} - t'_{\rho, \sigma} \right\rvert^{2},
        \end{equation*}
        then
        \begin{myenum}
            \item \(\lim\limits_{\rho, \sigma \to \infty} \left\lvert \rho \left( s_{\rho, \sigma} - s'_{\rho, \sigma} \right) - \delta \right\rvert^{2} + \sigma \left\lvert t_{\rho, \sigma} - t'_{\rho, \sigma} \right\rvert^{2} = 0\);
            \item \(\lim\limits_{\rho, \sigma \to \infty} M_{\rho, \sigma} = w \left( \ovs, \ovt \right) - v \left( \ovs, \ovt \right) = \max\limits_{(s, t) \in \ovQcal_{\gamma, T}} \{w(s, t) - v(s, t)\}\) whenever \(\left( \ovs, \ovt \right)\) is a limit point of \((s_{\rho, \sigma}, t_{\rho, \sigma})\) as \(\rho, \sigma \to \infty\).
        \end{myenum}
    \end{lem}

    \begin{theo}\label{loccomp}
        Fixed \(T > 0\), let \(W\) and \(V\) be continuous sub and supersolution to~\zcref{eq:locHJ} in \(\Qcal_{\gamma, T}\), respectively, and assume that one of the following holds true:
        \begin{myenum}
            \item \(W \le V\) on \(\partial_{p} \Qcal_{\gamma, T}\);
            \item \(W \le V\) on \(\partial_{p}^{+} \Qcal_{\gamma, T}\) and, for any \(t > 0\) and constrained \(\Crm^{1}\) subtangent to \(V\) at \((0, t)\),
            \begin{equation}\label{eq:loccomp.1}
                \difcp{\varphi(0, t)}{t} + H_{\gamma}(0, t, \difcp{\varphi}{s}) \ge 0.
            \end{equation}
        \end{myenum}
        Then \(W \le V\) on \(\ovQcal_{\gamma, T}\).
    \end{theo}
    \begin{proof}
        Here we will only prove the assertion for the second case, since the proof of the first one is almost identical. \\
        Possibly replacing \(W\) with
        \begin{equation*}
            (s, t) \longmapsto W(s, t) - \frac{\varepsilon}{T - t} \qquad \text{for some \(\varepsilon > 0\)},
        \end{equation*}
        we assume that \(W(s, T) < V(s, T)\) for any \(s \in [0, \lvert \gamma \rvert]\) and that \(W\) is a subsolution to
        \begin{equation*}
            \difcp{U(s, t)}{t} + H_{\gamma}(s, t, \difcp{U}{s}) \le - \frac{\varepsilon}{T^{2}} \qquad \text{in \(\Qcal_{\gamma, T}\)}.
        \end{equation*}
        We will show that under our assumptions \(W \le V\) on \(\ovQcal_{\gamma, T}\), which will prove the claim because \(\varepsilon\) is arbitrary. \\
        We proceed by contradiction, assuming that
        \begin{equation*}
            \max_{(s, t) \in \ovQcal_{\gamma, T}} \{W(s, t) - V(s, t)\} > 0.
        \end{equation*}
        Fixed \(\delta \in (0, \lvert \gamma \rvert)\), we define for \(\rho, \sigma \in \Rds^{+}\)
        \begin{equation*}
            M_{\rho, \sigma} \coloneqq \max_{(s, t), (s', t') \in \ovQcal_{\gamma, T}} \left\{ W(s, t) - V \left( s', t' \right) - \frac{1}{2} \left\lvert \rho \left( s - s' \right) - \delta \right\rvert^{2} - \frac{\sigma}{2} \left\lvert t - t' \right\rvert^{2} \right\},
        \end{equation*}
        then, for any \(\rho\) and \(\sigma\) big enough, there are \((s_{\rho, \sigma}, t_{\rho, \sigma}), \left( s'_{\rho, \sigma}, t_{\rho, \sigma}' \right) \in \ovQcal_{\gamma, T}\) so that
        \begin{align*}
            M_{\rho, \sigma} & = W(s_{\rho, \sigma}, t_{\rho, \sigma}) - V \left( s'_{\rho, \sigma}, t'_{\rho, \sigma} \right) - \frac{1}{2} \left\lvert \rho \left( s_{\rho, \sigma} - s'_{\rho, \sigma} \right) - \delta \right\rvert^{2} - \frac{\sigma}{2} \left\lvert t_{\rho, \sigma} - t'_{\rho, \sigma} \right\rvert^{2} \\
            & \ge \max_{(s, t) \in \ovQcal_{\gamma, T}} \{W(s, t) - V(s, t)\} > 0.
        \end{align*}
        As a consequence of \zcref{alphamax} we further have that, for any \(\rho\) and \(\sigma\) big enough, \((s_{\rho, \sigma}, t_{\rho, \sigma}) \in \Qcal_{\gamma, T}\) and \(\left( s'_{\rho, \sigma}, t_{\rho, \sigma}' \right) \in [0, \lvert \gamma \rvert) \times (0, T)\). Setting
        \begin{equation*}
            \varphi_{\rho, \sigma}(s, t) \coloneqq W(s_{\rho, \sigma}, t_{\rho, \sigma}) - \frac{1}{2} \lvert \rho(s_{\rho, \sigma} - s) - \delta \rvert^{2} - \frac{\sigma}{2} \lvert t_{\rho, \sigma} - t \rvert^{2} - M_{\rho, \sigma},
        \end{equation*}
        we have that \(\varphi_{\rho, \sigma}\) is a (constrained, if \(s'_{\rho, \sigma} = 0\)) \(\Crm^{1}\) subtangent to \(V\) at \(\left( s'_{\rho, \sigma}, t'_{\rho, \sigma} \right)\), therefore the supersolution property of \(V\) and~\zcref{eq:loccomp.1} imply
        \begin{equation}\label{eq:loccomp1}
            \sigma \left( t_{\rho, \sigma} - t_{\rho, \sigma}' \right) + H_{\gamma} \left( s_{\rho, \sigma}', t_{\rho, \sigma}', \rho^{2} \left( s_{\rho, \sigma} - s_{\rho, \sigma}' \right) - \rho \delta \right) \ge 0.
        \end{equation}
        Similarly,
        \begin{equation*}
            \psi_{\rho, \sigma}(s, t) \coloneqq V \left( s'_{\rho, \sigma}, t'_{\rho, \sigma} \right) + \frac{1}{2} \left\lvert \rho \left( s - s'_{\rho, \sigma} \right) - \delta \right\rvert^{2} + \frac{\sigma}{2} \left\lvert t - t'_{\rho, \sigma} \right\rvert^{2} + M_{\rho, \sigma}
        \end{equation*}
        is a \(\Crm^{1}\) supertangent to \(W\) at \((s_{\rho, \sigma}, t_{\rho, \sigma})\)---which we assume is not constrained since \(s_{\rho, \sigma} > 0\) for any \(\rho\) big enough---hence
        \begin{equation}\label{eq:loccomp2}
            \sigma \left( t_{\rho, \sigma} - t_{\rho, \sigma}' \right) + H_{\gamma} \left( s_{\rho, \sigma}, t_{\rho, \sigma}, \rho^{2} \left( s_{\rho, \sigma} - s_{\rho, \sigma}' \right) - \rho \delta \right) \le - \frac{\varepsilon}{T^{2}}.
        \end{equation}
        For any \(\rho\) and \(\sigma\) big enough~\zcref{eq:loccomp1,eq:loccomp2} then yield
        \begin{equation}\label{eq:loccomp3}
            H_{\gamma} \left( s_{\rho, \sigma}', t_{\rho, \sigma}', \rho^{2} \left( s_{\rho, \sigma} - s_{\rho, \sigma}' \right) - \rho \delta \right) - H_{\gamma} \left( s_{\rho, \sigma}, t_{\rho, \sigma}, \rho^{2} \left( s_{\rho, \sigma} - s_{\rho, \sigma}' \right) - \rho \delta \right) \ge \frac{\varepsilon}{T^{2}}.
        \end{equation}
        We further notice that by~\zcref{eq:loccomp2}
        \begin{equation}\label{eq:loccomp4}
            H_{\gamma} \left( s_{\rho, \sigma}, t_{\rho, \sigma}, \rho^{2} \left( s_{\rho, \sigma} - s_{\rho, \sigma}' \right) - \rho \delta \right) \le \sigma \left\lvert t_{\rho, \sigma}' - t_{\rho, \sigma} \right\rvert - \frac{\varepsilon}{T^{2}},
        \end{equation}
        thus the coercivity of \(H_{\gamma}\) shows that fixed \(\sigma\) big enough there is a constant \(\ell_{\sigma}\) such that
        \begin{equation*}
            \left\lvert \rho^{2} \left( s_{\rho, \sigma} - s_{\rho, \sigma}' \right) - \rho \delta \right\rvert \le \ell_{\sigma} \qquad \text{for every \(\rho\) big enough}.
        \end{equation*}
        It follows that if we fix \(\sigma\) and send \(\rho\) to \(\infty\), then \(\rho^{2} \left( s_{\rho, \sigma} - s_{\rho, \sigma}' \right) - \rho \delta\) converges, up to subsequences, to a \(\mu_{\sigma}\) with \(\lvert \mu_{\sigma} \rvert \le \ell_{\sigma}\). Likewise, \(t_{\rho, \sigma}\), \(t_{\rho, \sigma}'\) and \(s_{\rho, \sigma}\) converge, up to subsequences, to some \(t_{\sigma}\), \(t_{\sigma}'\) and \(s_{\sigma}\), respectively. Exploiting \zcref{alphamax}, we also get \(s'_{\rho, \sigma} \to s_{\sigma}\). We apply these facts to~\zcref{eq:loccomp3} to get
        \begin{align*}
            \frac{\varepsilon}{T^{2}} & \le H_{\gamma} \left( s_{\sigma}, t_{\sigma}', \mu_{\sigma} \right) - H_{\gamma}(s_{\sigma}, t_{\sigma}, \mu_{\sigma}) \le \left( \alpha_{T} \frac{H_{\gamma}(s_{\sigma}, t_{\sigma}, \mu_{\sigma})}{1 + \lvert \mu_{\sigma} \rvert} + \beta_{T} \right) \left\lvert t_{\sigma}' - t_{\sigma} \right\rvert \\
            & \le \frac{\alpha_{T} \sigma \lvert t_{\sigma}' - t_{\sigma} \rvert^{2}}{1 + \lvert \mu_{\sigma} \rvert} - \frac{\alpha_{T} \varepsilon \lvert t_{\sigma}' - t_{\sigma} \rvert}{T^{2} (1 + \lvert \mu_{\sigma} \rvert)} + \beta_{T} \lvert t_{\sigma}' - t_{\sigma} \rvert,
        \end{align*}
        where last inequalities are a consequence of \zcref{condHtlip,eq:loccomp4}. Since \(\sigma \lvert t_{\sigma}' - t_{\sigma} \rvert^{2} \to 0\) as \(\sigma \to \infty\) by \zcref{alphamax}, we have a contradiction.
    \end{proof}

    We now focus on the maximal subsolutions to~\zcref{eq:locHJ}. Given a continuous map \(U_{0} \colon \partial_{p}^{+} \Qcal_{\gamma} \to \Rds\), we consider the value function
    \begin{equation}\label{eq:locvf}
        U(s, t) \coloneqq \inf \left\{ U_{0}(\eta(t_{0}), t_{0}) + \int_{t_{0}}^{t} L_{\gamma}(\eta, \dot{\eta}, \tau) \dl \tau \right\},
    \end{equation}
    where the infimum is taken over the \(t_{0} \in [0, t]\) and the curves \(\eta \colon [t_{0}, t] \to [0, \lvert \gamma \rvert]\) with \(\eta(t) = s\) and \(\eta(t_{0}) \in \partial_{p}^{+} \Qcal_{\gamma}\).

    \begin{rem}\label{locmincurve}
        \zcref[S]{equivLagprops,minactlsc,mincurvelip} show that the infimum in~\zcref{eq:locvf} is achieved by a Lipschitz continuous curve.
    \end{rem}

    \begin{prop}\label{locvfcont}
        The value function \(U\) defined in~\zcref{eq:locvf} is continuous.
    \end{prop}
    \begin{proof}
        We consider the map
        \begin{equation*}
            \Phi_{\gamma}(s_{1}, t_{1}, s_{2}, t_{2}) \coloneqq \inf \left\{ \int_{t_{0}}^{t} L_{\gamma}(\eta, \dot{\eta}, \tau) \dl \tau : \text{\(\eta\) is a curve with \(\eta(t_{1}) = s_{1}\) and \(\eta(t_{2}) = s_{2}\)} \right\}
        \end{equation*}
        with the additional assumption
        \begin{equation*}
            \Phi_{\gamma}(s_{1}, t_{1}, s_{2}, t_{2}) \coloneqq \infty \qquad \text{whenever \(t_{1} = t_{2}\) and \(s_{1} \ne s_{2}\)}.
        \end{equation*}
        Given a sequence \(\{(s_{n}, t_{n})\}_{n \in \Nds} \subset \ovQcal_{\gamma}\) converging to some \((s, t)\), there is also a sequence \(\{(s_{n}', t_{n}')\}_{n \in \Nds} \subset \partial^{+}_{p} \Qcal_{\gamma}\) with \(t_{n}' \le t_{n}\) and an \((s_{0}, t_{0}) \in \partial^{+}_{p} \Qcal_{\gamma}\) with \(t_{0} \le t\) such that
        \begin{equation*}
            U(s_{n}, t_{n}) = U_{0} \left( s_{n}', t_{n}' \right) + \Phi_{\gamma} \left( s_{n}', t_{n}', s_{n}, t_{n} \right) \qquad \text{and} \qquad U(s, t) = U_{0}(s_{0}, t_{0}) + \Phi_{\gamma}(s_{0}, t_{0}, s, t).
        \end{equation*}
        Because \((s, t)\) and \(\{(s_{n}, t_{n})\}\) are arbitrary, it is enough to show that
        \begin{equation}\label{eq:locvfcont1}
            \lim_{n \to \infty} U(s_{n}, t_{n}) = U(s, t).
        \end{equation}
        According to \zcref{equivLagprops,minactlip}, \(\Phi_{\gamma}\) is continuous at any point with \(t_{1} < t_{2}\); therefore, \(U\) is continuous at \((s, t)\) whenever \(t > t_{0}\). It remains to prove the \(t = t_{0}\) case. \\
        First we observe that \(\{(s_{n}', t_{n}')\}\) must converge, up to subsequences, to an \((s_{0}', t_{0}') \in \partial_{p}^{+} \Qcal_{\gamma}\) with \(t_{0}' \le t\). Since \(\Phi_{\gamma}\) is lower semicontinuous by \zcref{equivLagprops,minactlsc}, we then have
        \begin{equation}\label{eq:locvfcont2}
            \liminf_{n \to \infty} U(s_{n}, t_{n}) \ge U_{0} \left( s_{0}', t_{0}' \right) + \Phi_{\gamma} \left( s_{0}', t_{0}', s, t \right) \ge U(s, t).
        \end{equation}
        Next we assume that \(t = 0\), hence \(t_{0}' = 0\). In this case, for any \(n\) big enough,
        \begin{align*}
            U(s_{n}, t_{n}) & = U_{0} \left( s_{n}', t_{n}' \right) + \Phi_{\gamma} \left( s_{n}', t_{n}', s_{n}, t_{n} \right) \le U_{0}(s_{n}, 0) + \Phi_{\gamma}(s_{n}, 0, s_{n}, t_{n}) \\
            & \le U_{0}(s_{n}, 0) + \int_{0}^{t_{n}} \Theta_{1}(0) \dl \tau = U_{0}(s_{n}, 0) + t_{n} \Theta_{1}(0),
        \end{align*}
        which implies that
        \begin{equation}\label{eq:locvfcont3}
            \limsup_{n \to \infty} U(s_{n}, t_{n}) \le U(s, t) \qquad \text{if \(t = 0\)}.
        \end{equation}
        If instead \(t > 0\), then \(s = \lvert \gamma \rvert\) and, for any \(n\) big enough,
        \begin{align*}
            U(s_{n}, t_{n}) & = U_{0} \left( s_{n}', t_{n}' \right) + \Phi_{\gamma} \left( s_{n}', t_{n}', s_{n}, t_{n} \right) \\
            & \le U_{0}(\lvert \gamma \rvert, t_{n} - (\lvert \gamma \rvert - s_{n})) + \Phi_{\gamma}(\lvert \gamma \rvert, t_{n} - (\lvert \gamma \rvert - s_{n}), s_{n}, t_{n}) \\
            & \le U_{0}(\lvert \gamma \rvert, t_{n} - (\lvert \gamma \rvert - s_{n})) + \int_{t_{n} - (\lvert \gamma \rvert - s_{n})}^{t_{n}} \Theta_{t + 1}(1) \dl \tau \\
            & = U_{0}(\lvert \gamma \rvert, t_{n} - (\lvert \gamma \rvert - s_{n})) + (\lvert \gamma \rvert - s_{n}) \Theta_{t + 1}(1),
        \end{align*}
        which proves that
        \begin{equation}\label{eq:locvfcont4}
            \limsup_{n \to \infty} U(s_{n}, t_{n}) \le U(s, t) \qquad \text{if \(t > 0\) and \(s = \lvert \gamma \rvert\)}.
        \end{equation}
        Finally, \zcref{eq:locvfcont2,eq:locvfcont3,eq:locvfcont4} yield~\zcref{eq:locvfcont1} when \(t = t_{0}\).
    \end{proof}

    \begin{theo}\label{locmaxsubsol}
        The value function \(U\) defined in~\zcref{eq:locvf} is the maximal continuous subsolution to~\zcref{eq:locHJ} less than or equal to \(U_{0}\) on \(\partial_{p}^{+} \Qcal_{\gamma}\).
    \end{theo}
    \begin{proof}
        \zcref[S]{locvfcont} shows that \(U\) is continuous, therefore \zcref{locmincurve,subsolloc} yield that \(U\) is the maximal subsolution to~\zcref{eq:locHJ} such that \(U \le U_{0}\) on \(\partial_{p}^{+} \Qcal_{\gamma}\).
    \end{proof}

    \section{A Related Semidiscrete Problem}\label{sec:disc}

    We introduce here a semidiscrete problem strictly related to~\zcref{eq:globHJ}.

    Let
    \begin{equation*}
        \Fcal \coloneqq \Crm \left( (\Gamma \times \{0\}) \cup \left( \Vbf \times \Rds^{+} \right) \right),
    \end{equation*}
    then define for any \(\gamma \in \Ecal\) the operator \(F_{\gamma} \colon \Fcal \to \Crm \left( \ovQcal_{\gamma} \right)\) which associates to a function \(u \in \Fcal\) the maximal continuous subsolution \(W\) to~\zcref{eq:locHJ} with \(W(s, t) \le u(\gamma(s), t)\) for any \((s, t) \in \partial_{p}^{+} \Qcal_{\gamma}\). Thanks to \zcref{locmaxsubsol}, we know that
    \begin{equation*}
        F_{\gamma}[u](s, t) = \inf \left\{ u(\gamma(\eta(t_{0})), t_{0}) + \int_{t_{0}}^{t} L_{\gamma}(\eta, \dot{\eta}, \tau) \dl \tau \right\},
    \end{equation*}
    where the infimum is taken over the \(t_{0} \in [0, t]\) and the curves \(\eta \colon [t_{0}, t] \to [0, \lvert \gamma \rvert]\) with \(\eta(t) = s\) and \(\eta(t_{0}) \in \partial_{p}^{+} \Qcal_{\gamma}\). Next we set
    \begin{equation*}
        F_{x}[u](t) = \min_{\gamma \in \Ecal_{x}} F_{\gamma}[u](0, t) \qquad \text{for \(t \in \Rds^{+}\)}.
    \end{equation*}

    We summarize the main properties of \(F_{x}\) in the next \zcref[noref]{Fxprop}.

    \begin{lem}\label{Fxprop}
        Given \(u \in \Fcal\) and \(x \in \Vbf\),
        \begin{mylist}
            \item \(F_{x}[u] \colon \Rds^{+} \to \Rds\) is continuous;
            \item \(F_{x}[u](0) = u(x, 0)\);
            \item \(F_{x}[u + a] = F_{x}[u] + a\) for any \(a \in \Rds\);
            \item let \(v \in \Fcal\) and \(T > 0\) be so that
            \begin{equation*}
                v(\gamma(s), t) \ge u(\gamma(s), t) \qquad \text{for each \(\gamma \in \Ecal_{x}\) and \((s, t) \in \partial_{p}^{+} \Qcal_{\gamma, T}\)},
            \end{equation*}
            then \(F_{x}[v](t) \ge F_{x}[u](t)\) for every \(t \in [0, T]\).
        \end{mylist}
    \end{lem}

    We also define the operator \(G \colon \Crm(\Rds^{+}) \times \Crm(\Rds^{+}) \to \Rds\) via
    \begin{equation}\label{eq:Gdef}
        G[\psi, c](t) = \min_{r \in [0, t]} \left\{ \psi(r) + \int_{r}^{t} c(\tau) \dl \tau \right\} \qquad \text{for \(\psi, c \in \Crm \left( \Rds^{+} \right)\), \(t \in \Rds^{+}\)}.
    \end{equation}
    Loosely speaking, this operator allows taking into account the constraint given in~\zcref{eq:globHJ} by the flux limiter.

    \begin{lem}\label{Gismin}
        We have
        \begin{equation}\label{eq:Gismin.1}
            G[\psi, c](t) = \min_{r \in [0, t]} \left\{ G[\psi, c](r) + \int_{r}^{t} c(\tau) \dl \tau \right\} \qquad \text{for any \(\psi, c \in \Crm \left( \Rds^{+} \right)\) and \(t \in \Rds^{+}\)}.
        \end{equation}
    \end{lem}
    \begin{proof}
        From the definition of \(G\) we have that, fixed \(\psi, c \in \Crm(\Rds^{+})\), \(t \in \Rds^{+}\) and \(r \in [0, t]\), there is an \(r' \le r\) such that
        \begin{equation*}
            G[\psi, c](r) + \int_{r}^{t} c(\tau) \dl \tau = \psi(r') + \int_{r'}^{r} c(\tau) \dl \tau + \int_{r}^{t} c(\tau) \dl \tau = \psi(r') + \int_{r'}^{t} c(\tau) \dl \tau \ge G[\psi, c](t).
        \end{equation*}
        Since \(\psi\), \(c\), \(t\) and \(r\) are arbitrary, this proves
        \begin{equation*}
            G[\psi, c](t) \le \min_{r \in [0, t]} \left\{ G[\psi, c](r) + \int_{r}^{t} c(\tau) \dl \tau \right\} \qquad \text{for any \(\psi, c \in \Crm \left( \Rds^{+} \right)\) and \(t \in \Rds^{+}\)}.
        \end{equation*}
        On the other hand, it is apparent that
        \begin{equation*}
            G[\psi, c](t) \ge \min_{r \in [0, t]} \left\{ G[\psi, c](r) + \int_{r}^{t} c(\tau) \dl \tau \right\} \qquad \text{for any \(\psi, c \in \Crm \left( \Rds^{+} \right)\) and \(t \in \Rds^{+}\)},
        \end{equation*}
        thus \zcref{eq:Gismin.1} holds true.
    \end{proof}

    \begin{lem}\label{Gmax}
        Fixed \(\psi, c \in \Crm(\Rds^{+})\), \(G[\psi, c]\) is the maximal continuous function in \(\Rds^{+}\) less than or equal to \(\psi\) satisfying
        \begin{equation}\label{eq:Gmax.1}
            \difcp{\varphi(t)}{} \le c(t) \qquad \text{for any \(t \in (0, \infty)\) and \(\Crm^{1}\) supertangent \(\varphi\) to \(G[\psi, c]\) at \(t\)}.
        \end{equation}
    \end{lem}
    \begin{proof}
        The map \(t \mapsto G[\psi, c](t)\) is lower semicontinuous, since it is the minimum of a collection of continuous functions. Setting \(r_{0}\) as the time realizing the minimum in~\zcref{eq:Gdef} for a fixed \(t \in \Rds^{+}\), we further have
        \begin{equation*}
            \limsup_{r \to t} G[\psi, c](r) \le \lim_{r \to t} \left( \psi(r_{0} \wedge r) + \int_{r_{0} \wedge r}^{r} c(\tau) \dl \tau \right) = \psi(r_{0}) + \int_{r_{0}}^{t} c(\tau) \dl \tau = G[\psi, c](t),
        \end{equation*}
        where \(r_{0} \wedge r\) is the minimum of \(r_{0}\) and \(r\). This shows that \(G[\psi, c]\) is also upper semicontinuous and thus continuous. Furthermore, \zcref{Gismin} yields that for any \(t \in (0, \infty)\), \(\Crm^{1}\) supertangent \(\varphi\) to \(G[\psi, c]\) at \(t\) and \(\delta > 0\) small enough
        \begin{equation*}
            \varphi(t) - \varphi(t - \delta) \le G[\psi, c](t) - G[\psi, c](t - \delta) \le \int_{t - \delta}^{t} c(\tau) \dl \tau,
        \end{equation*}
        hence~\zcref{eq:Gmax.1} holds true. \\
        Now let \(f \le \psi\) be a continuous function in \(\Rds^{+}\) satisfying~\zcref{eq:Gmax.1}. \zcref[S]{suptanmon} implies that
        \begin{equation*}
            f(t) \le f(r) + \int_{r}^{t} c(\tau) \dl \tau \le \psi(r) + \int_{r}^{t} c(\tau) \dl \tau \qquad \text{for any \(t > 0\) and \(r \in [0, t]\)},
        \end{equation*}
        which proves that \(f \le G[\psi, c]\) and concludes the proof.
    \end{proof}

    \begin{prop}\label{upboundG}
        Fixed \(\psi, c \in \Crm(\Rds^{+})\), let \(f \in \Crm(\Rds^{+})\) be such that \(f(0) \ge \psi(0)\) and \(f(t) \ge \psi(t)\) whenever there is a \(\Crm^{1}\) subtangent \(\varphi\) to \(f\) at \(t\) with \(\difcp{\varphi(t)}{} < c(t)\). Then
        \begin{equation*}
            f \ge G[\psi, c] \qquad \text{on \(\Rds^{+}\)}.
        \end{equation*}
    \end{prop}
    \begin{proof}
        Given \(t \in \Rds^{+}\), let \(E\) be the set made up of the \(r \in (0, t)\) such that \(f\) admits a \(\Crm^{1}\) subtangent \(\varphi\) at \(r\) with \(\difcp{\varphi(r)}{} < c(r)\). We define
        \begin{equation*}
            r_{0} \coloneqq
            \begin{dcases}
                \sup E & \text{if \(E \ne \emptyset\)}, \\
                0 & \text{if \(E = \emptyset\)}.
            \end{dcases}
        \end{equation*}
        We point out that \(\difcp{\varphi(r)}{} \ge c(r)\) whenever \(\varphi\) is a subtangent to \(f\) at \(r \in (r_{0}, t)\), while \(f(r_{0}) \ge \psi(r_{0})\) by our assumptions. \zcref[S]{suptanmon} thus yields
        \begin{equation*}
            f(t) \ge f(r_{0}) + \int_{r_{0}}^{t} c(\tau) \dl \tau \ge \psi(r_{0}) + \int_{r_{0}}^{t} c(\tau) \dl \tau \ge G[\psi, c](t).
        \end{equation*}
        This concludes the proof.
    \end{proof}

    \subsection{Definition of the Problem and Comparison Result}

    Given a flux limiter \(x \mapsto c_{x}\) on \(\Vbf\), we consider the semidiscrete equation
    \begin{equation}\label{eq:disc}\tag{Discr}
        u(x, t) = G[F_{x}[u], c_{x}](t) \qquad \text{on \(\Vbf \times (0, \infty)\)}.
    \end{equation}
    A \emph{solution} to~\zcref{eq:disc} is any \(u \in \Fcal\) which satisfies pointwise the identity in~\zcref{eq:disc}. A \emph{subsolution} (resp.\ \emph{supersolution}) is any \(w \in \Fcal\) such that
    \begin{equation*}
        w(x, t) \le G[F_{x}[w], c_{x}](t) \quad ( \text{resp.\ } w(x, t) \ge G[F_{x}[w], c_{x}](t)) \qquad \text{for any \((x, t) \in \Vbf \times (0, \infty)\)}.
    \end{equation*}
    If
    \begin{equation*}
        w(x, t) < G[F_{x}[w], c_{x}](t) \qquad \text{for any \((x, t) \in \Vbf \times (0, \infty)\)},
    \end{equation*}
    we will say that \(w\) is a strict subsolution.

    \begin{lem}\label{strsubdisc}
        Let \(w \in \Fcal\) be a subsolution to~\zcref{eq:disc}, then \((s, t) \mapsto w(s, t) - \varepsilon t\) is a strict subsolution for any \(\varepsilon > 0\).
    \end{lem}
    \begin{proof}
        We start by fixing \(\varepsilon > 0\) and \(\ovw(s, t) \coloneqq w(s, t) - \varepsilon t\). It follows from \zcref{locmincurve} that, for each \(\gamma \in \Ecal\) and \(t \in (0, \infty)\), there are \(t_{0} < t\) and a curve \(\eta \colon [t_{0}, t] \to [0, \lvert \gamma \rvert]\) so that
        \begin{align*}
            F_{\gamma}[\ovw](0, t) & = w(\gamma(\eta(t_{0})), t_{0}) - \varepsilon t_{0} + \int_{t_{0}}^{t} L_{\gamma}(\eta, \dot{\eta}, \tau) \dl \tau \\
            & > w(\gamma(\eta(t_{0})), t_{0}) - \varepsilon t + \int_{t_{0}}^{t} L_{\gamma}(\eta, \dot{\eta}, \tau) \dl \tau \ge F_{\gamma}[w](0, t) - \varepsilon t,
        \end{align*}
        which in turn implies
        \begin{equation}\label{eq:strsubdisc1}
            F_{x}[w](t) - \varepsilon t < F_{x}[\ovw](t) \qquad \text{for every \((x, t) \in \Vbf \times (0, \infty)\)}.
        \end{equation}
        We deduce from~\zcref{eq:strsubdisc1} and \(w\) being a subsolution that, for any \((x, t) \in \Vbf \times (0, \infty)\) and \(r \in (0, t]\),
        \begin{equation}\label{eq:strsubdisc2}
            \begin{aligned}
                w(x, t) - \varepsilon t & \le G[F_{x}[w], c_{x}](t) - \varepsilon t \le F_{x}[w](r) + \int_{r}^{t} c_{x}(\tau) \dl \tau - \varepsilon t \\
                & \le F_{x}[w](r) + \int_{r}^{t} c_{x}(\tau) \dl \tau - \varepsilon r < F_{x}[\ovw](r) + \int_{r}^{t} c_{x}(\tau) \dl \tau.
            \end{aligned}
        \end{equation}
        Similarly, \zcref{Fxprop} shows that
        \begin{equation}\label{eq:strsubdisc3}
            \begin{aligned}
                w(x, t) - \varepsilon t & \le G[F_{x}[w], c_{x}](t) - \varepsilon t \le F_{x}[w](0) + \int_{0}^{t} c_{x}(\tau) \dl \tau - \varepsilon t \\
                & < F_{x}[w](0) + \int_{0}^{t} c_{x}(\tau) \dl \tau = F_{x}[\ovw](0) + \int_{0}^{t} c_{x}(\tau) \dl \tau.
            \end{aligned}
        \end{equation}
        By combining \zcref{eq:strsubdisc2,eq:strsubdisc3} we obtain
        \begin{equation*}
            \ovw < G[F_{x}[\ovw], c_{x}] \qquad \text{on \(\Vbf \times (0, \infty)\)},
        \end{equation*}
        which proves the claim.
    \end{proof}

    \begin{theo}\label{disccomp}
        Let \(w\) and \(v\) be continuous sub and supersolution to~\zcref{eq:disc}, respectively, such that
        \begin{equation*}
            w(x, 0) \le v(x, 0) \qquad \text{for all \(x \in \Gamma\)},
        \end{equation*}
        then \(w \le v\) in \(\Vbf \times \Rds^{+}\).
    \end{theo}
    \begin{proof}
        Possibly replacing \(w\) with \((x, t) \mapsto w(x, t) - \varepsilon t\) for some \(\varepsilon > 0\) and bearing in mind \zcref{strsubdisc}, we can assume without loss of generality that \(w\) is a strict subsolution. We proceed by contradiction, assuming that
        \begin{equation*}
            w(x_{0}, t_{0}) > v(x_{0}, t_{0}) \qquad \text{for some \((x_{0}, t_{0}) \in \Vbf \times (0, \infty)\)}.
        \end{equation*}
        First we denote with \(t^{*}\) a maximizer of \(t \mapsto w(x_{0}, t) - v(x_{0}, t)\) in \([0, t_{0}]\) and define
        \begin{equation}\label{eq:disccomp1}
            a \coloneqq w(x_{0}, t^{*}) - v(x_{0}, t^{*}) > 0,
        \end{equation}
        which implies
        \begin{equation}\label{eq:disccomp2}
            v(x_{0}, t) + a \ge w(x_{0}, t) \qquad \text{for all \(t \in [0, t_{0}]\)}.
        \end{equation}
        Next we set \(r \in [0, t^{*}]\) such that
        \begin{equation}\label{eq:disccomp3}
            v(x_{0}, t^{*}) \ge G[F_{x_{0}}[v], c_{x_{0}}](t^{*}) = F_{x_{0}}[v](r) + \int_{r}^{t^{*}} c_{x_{0}}(\tau) \dl \tau,
        \end{equation}
        while the strict subsolution property of \(w\) yields
        \begin{equation}\label{eq:disccomp4}
            w(x_{0}, t^{*}) < G[F_{x_{0}}[w], c_{x_{0}}](t^{*}) \le F_{x_{0}}[w](r) + \int_{r}^{t^{*}} c_{x_{0}}(\tau) \dl \tau.
        \end{equation}
        Finally, combining \zcref{eq:disccomp2,eq:disccomp3,eq:disccomp4,Fxprop} we get
        \begin{align*}
            v(x_{0}, t^{*}) + a & \ge F_{x_{0}}[v](r) + \int_{r}^{t^{*}} c_{x_{0}}(\tau) \dl \tau + a = F_{x_{0}}[v + a](r) + \int_{r}^{t^{*}} c_{x_{0}}(\tau) \dl \tau \\
            & \ge F_{x_{0}}[w](r) + \int_{r}^{t^{*}} c_{x_{0}}(\tau) \dl \tau > w(x_{0}, t^{*}),
        \end{align*}
        in contradiction with~\zcref{eq:disccomp1}.
    \end{proof}

    \subsection{Links between semidiscrete and Hamilton--Jacobi equations}

    We proceed linking \zcref{eq:disc} to~\zcref{eq:globHJ}.

    \begin{prop}\label{subHJ2disc}
        Given a continuous subsolution to~\zcref{eq:globHJ} with flux limiter \(c_{x}\), its restriction \(w\) to \((\Gamma \times \{0\}) \cup (\Vbf \times \Rds^{+})\) is a subsolution to~\zcref{eq:disc}.
    \end{prop}
    \begin{proof}
        We know from \zcref{defsol}\zcref[noname]{en:subsol2} that, for any \(x \in \Vbf\), \(t \in (0, \infty)\) and \(\Crm^{1}\) supertangent \(\psi\) to \(w(x, \cdot)\) at \(t\),
        \begin{equation*}
            \difcp{\psi(t)}{} \le c_{x}(t),
        \end{equation*}
        while the subsolution property of \(w\) and the definition of \(F_{x}\) yield
        \begin{equation*}
            w(x, t) \le F_{x}[w](t) \qquad \text{for any \((x, t) \in \Vbf \times (0, \infty)\)}.
        \end{equation*}
        \zcref[S]{Gmax} then implies
        \begin{equation*}
            w(x, t) \le G[F_{x}[w], c_{x}](t) \qquad \text{for any \((x, t) \in \Vbf \times (0, \infty)\)},
        \end{equation*}
        as was claimed.
    \end{proof}

    \begin{prop}\label{supHJ2disc}
        Given a continuous supersolution to~\zcref{eq:globHJ} with flux limiter \(c_{x}\), its restriction \(v\) to \((\Gamma \times \{0\}) \cup (\Vbf \times \Rds^{+})\) is a supersolution to~\zcref{eq:disc}.
    \end{prop}
    \begin{proof}
        We fix a vertex \(x\). According to \zcref{upboundG}, to prove that
        \begin{equation*}
            v(x, t) \ge G[F_{x}[v], c_{x}](t) \qquad \text{for all \(t \in (0, \infty)\)},
        \end{equation*}
        it is enough to show that, given a \(\ovt \in (0, \infty)\) at which \(v\) admits a \(\Crm^{1}\) subtangent \(\varphi\) with \(\difcp{\varphi(\ovt)}{} < c_{x}(\ovt)\), one has
        \begin{equation}\label{eq:supHJ2disc1}
            v \left( x, \ovt \right) \ge F_{x}[v] \left( \ovt \right).
        \end{equation}
        We proceed by contradiction, assuming that~\zcref{eq:supHJ2disc1} is not true, i.e.,
        \begin{equation}\label{eq:supHJ2disc2}
            v \left( x, \ovt \right) < F_{x}[v] \left( \ovt \right).
        \end{equation}
        Fixed \(\gamma \in \Ecal_{x}\), we notice that by \zcref{eq:fldef} there is a \(\mu_{0} \in \Rds\) such that
        \begin{equation*}
            \difcp{\varphi \left( \ovt \right)}{} + H_{\gamma} \left( 0, \ovt, \mu_{0} \right) < c_{x} \left( \ovt \right) + H_{\gamma} \left( 0, \ovt, \mu_{0} \right) \le 0,
        \end{equation*}
        therefore, taking an \(A > 0\) big enough and setting
        \begin{equation*}
            \phi(s, t) \coloneqq \varphi(t) - A \left( t - \ovt \right)^{2} + \mu_{0} s \qquad \text{for \((s, t) \in \ovQcal_{\gamma}\)},
        \end{equation*}
        we have that there exist a \(T > \ovt\) and a \(\delta > 0\) such that
        \begin{equation}\label{eq:supHJ2disc3}
            \phi \left( 0, \ovt \right) = v \left( x, \ovt \right), \qquad \phi(s, t) \le v(\gamma(s), t) \quad \text{for any \((s, t) \in \partial_{p}^{-} \Qcal_{\gamma, T}\)}
        \end{equation}
        and
        \begin{equation}\label{eq:supHJ2disc4}
            \difcp{\phi(s, t)}{t} + H_{\gamma}(s, t, \difcp{\phi(s, t)}{s}) < 0 \qquad \text{for any \((s, t) \in [0, \delta] \times \left[ \ovt - \delta, \ovt + \delta \right]\)}.
        \end{equation}
        We will prove
        \begin{equation}\label{eq:supHJ2disc5}
            v(\gamma(s), t) \ge \phi(s, t) \qquad \text{for every \((s, t) \in [0, \delta] \times \left[ \ovt - \delta, \ovt + \delta \right]\)},
        \end{equation}
        which, according to~\zcref{eq:supHJ2disc3}, will show that \(\phi\) is a constrained \(\Crm^{1}\) subtangent to \((s, t) \mapsto v(\gamma(s), t)\) at \(\left( 0, \ovt \right)\). This will conclude the proof because~\zcref{eq:supHJ2disc4} is in contradiction with \zcref{defsol}\zcref[noname]{en:supsol2}. \\
        First, exploiting \zcref{eq:supHJ2disc2}, we further assume that, up to shrinking \(\delta\),
        \begin{equation}\label{eq:supHJ2disc6}
            v(\gamma(s), t) < F_{\gamma}[v](s, t) \qquad \text{for each \((s, t) \in [0, \delta] \times \left[ \ovt - \delta, \ovt + \delta \right]\)}.
        \end{equation}
        Then we set
        \begin{equation*}
            W(s, t) = \inf \left\{ v(\gamma(\eta(t_{0})), t_{0}) + \int_{t_{0}}^{t} L_{\gamma}(\eta, \dot{\eta}, \tau) \dl \tau \right\} \qquad \text{for \((s, t) \in \ovQcal_{\gamma}\)},
        \end{equation*}
        where the infimum is taken over the \(t_{0} \in [0, t]\) and the curves \(\eta \colon [t_{0}, t] \to [0, \lvert \gamma \rvert]\) with \(\eta(t) = s\) and \(\eta(t_{0}) \in \partial_{p} \Qcal_{\gamma}\). It is easy to check that \(W \equiv \min \{F_{\gamma}[v], F_{\wtgamma}[v]\}\). Additionally, \(W\) is the maximal subsolution to~\zcref{eq:locHJ} less than or equal to \(v(\gamma(s), t)\) on \(\partial_{p} \Qcal_{\gamma}\), thus \zcref{loccomp,eq:supHJ2disc6} yield
        \begin{equation}\label{eq:supHJ2disc7}
            v(\gamma(s), t) \ge W(s, t) \ge F_{\wtgamma}[v](s, t) \qquad \text{for each \((s, t) \in [0, \delta] \times \left[ \ovt - \delta, \ovt + \delta \right]\)}.
        \end{equation}
        Next we define
        \begin{equation*}
            \ovW(s, t) \coloneqq \varphi(t_{0}) + \int_{t_{0}}^{t} \min \{\difcp{\phi(s, \tau)}{t}, - H_{\gamma}(s, \tau, \mu_{0})\} \dl \tau + \mu_{0} s \qquad \text{for \((s, t) \in \ovQcal_{\gamma}\)},
        \end{equation*}
        which is a subsolution to~\zcref{eq:locHJ} so that \(\ovW \le \phi\) and, by~\zcref{eq:supHJ2disc4,eq:supHJ2disc3},
        \begin{align}
            \ovW(s, t) & \le v(\gamma(s), t) & & \text{for all \((s, t) \in \partial_{p}^{-} \Qcal_{\gamma, T}\)}, \nonumber \\
            \ovW(s, t) & = \phi(s, t) & & \text{for all \((s, t) \in [0, \delta] \times \left[ \ovt - \delta, \ovt + \delta \right]\)} \label{eq:supHJ2disc8}.
        \end{align}
        The maximality of \(F_{\wtgamma}[v]\) implies that
        \begin{equation*}
            \ovW(s, t) \le F_{\wtgamma}[v](s, t) \qquad \text{for any \((s, t) \in \ovQcal_{\gamma, T}\)},
        \end{equation*}
        which, together with~\zcref{eq:supHJ2disc7,eq:supHJ2disc8}, proves~\zcref{eq:supHJ2disc5}.
    \end{proof}

    We can finally prove a comparison result for~\zcref{eq:globHJ}.

    \begin{theo}\label{comp}
        Let \(w\) and \(v\) be continuous sub and supersolution to~\zcref{eq:globHJ} with flux limiter \(c_{x}\) such that
        \begin{equation*}
            w(x, 0) \le v(x, 0) \qquad \text{for any \(x \in \Gamma\)}.
        \end{equation*}
        Then \(w \le v\) in \(\Gamma \times \Rds^{+}\).
    \end{theo}
    \begin{proof}
        Applying \zcref{subHJ2disc,supHJ2disc,disccomp} to \(w\) and \(v\) we get that
        \begin{equation*}
            w(x, t) \le v(x, t) \qquad \text{for every \((x, t) \in \Vbf \times \Rds^{+}\)},
        \end{equation*}
        therefore \zcref{loccomp} and \zcref{defsol}\zcref[noname]{en:subsol1} and \zcref[noname]{en:supsol1} yield
        \begin{equation*}
            w(\gamma(s), t) \le v(\gamma(s), t) \qquad \text{for each \(\gamma \in \Ecal\) and \((s, t) \in \ovQcal_{\gamma}\)}.
        \end{equation*}
    \end{proof}

    \section{Representation Formula}\label{sec:solution}

    The analysis performed in the previous sections allows us to provide a Lax--Oleinik-type representation formula for solutions to~\zcref{eq:globHJ} and to prove the well-posedness of the problem.

    Given a map \(u_{0} \colon \Gamma \to \Rds\), we define
    \begin{equation}\label{eq:vf}
        u(x, t) \coloneqq \inf_{x_{0} \in \Gamma} \{u_{0}(x_{0}) + \Phi(x_{0}, 0, x, t)\} \qquad \text{for \((x, t) \in \Gamma \times \Rds^{+}\)}.
    \end{equation}

    \begin{prop}\label{vfcont}
        If \(u_{0}\) is uniformly continuous, the value function \(u\) defined in~\zcref{eq:vf} is uniformly continuous on \(\Gamma \times [0, T]\) for any \(T > 0\). In particular, \(u\) is continuous on \(\Gamma \times \Rds^{+}\).
    \end{prop}

    To prove this \zcref[noref]{vfcont} we need some preliminary results.

    \begin{lem}\label{optvf}
        Given \(T > \ovt > 0\), if \(u_{0}\) is uniformly continuous there exists a \(C > 0\) such that, for every \((x, t) \in \Gamma \times \left[ \ovt, T \right]\), the infimum in~\zcref{eq:vf} can be taken over the \(x_{0}\) with
        \begin{equation}\label{eq:optvf.1}
            d_{\Gamma}(x_{0}, x) \le Ct.
        \end{equation}
    \end{lem}

    This \zcref[noref]{optvf}, combined with \zcref{minactlsc}, shows that if \(u_{0}\) is uniformly continuous, then the infimum in~\zcref{eq:vf} is actually a minimum.

    \begin{proof}
        Let \(\omega\) be a concave modulus of continuity of \(u_{0}\), then there are two positive constant \(a\) and \(b\) so that
        \begin{equation}\label{eq:optvf1}
            \omega(r) \le ar + b \qquad \text{for all \(r \in \Rds^{+}\)}.
        \end{equation}
        Exploiting the coercivity of \(\ovvartheta_{T}\) in \zcref{gLagprop}, we further have that there is a \(B > 0\) such that
        \begin{equation*}
            L(z, q, t) \ge \ovvartheta_{T}(\lvert q \rvert) \ge (a + 1) \lvert q \rvert - B \qquad \text{for any \((z, q) \in \Trm \Gamma\) and \(t \in [0, T]\)};
        \end{equation*}
        thus, given an \(x_{0} \in \Gamma\) and a minimizing curve \(\zeta\) for \(\Phi(x_{0}, 0, x, t)\),
        \begin{equation}\label{eq:optvf2}
            \Phi(x_{0}, 0, x, t) = \int_{0}^{t} L \left( \zeta, \dot{\zeta}, \tau \right) \dl \tau \ge (a + 1) \int_{0}^{t} \left\lvert \dot{\zeta}(\tau) \right\rvert \dl \tau - Bt \ge (a + 1) d_{\Gamma}(x_{0}, x) - Bt.
        \end{equation}
        Additionally, the infimum in~\zcref{eq:vf} can be taken over the \(x_{0}\) with
        \begin{equation}\label{eq:optvf3}
            u_{0}(x_{0}) + \Phi(x_{0}, 0, x, t) \le u_{0}(x) + \Phi(x, 0, x, t) \le u_{0}(x) + \int_{0}^{t} L(x, 0, \tau) \dl \tau \le u_{0}(x) + t \ovTheta_{T}(0).
        \end{equation}
        Combining~\zcref{eq:optvf1,eq:optvf2,eq:optvf3} we finally get
        \begin{align*}
            (a + 1) d_{\Gamma}(x_{0}, x) & \le \Phi(x_{0}, 0, x, t) + Bt \le u_{0}(x) - u_{0}(x_{0}) + t \ovTheta_{T}(0) + Bt \\
            & \le a d_{\Gamma}(x_{0}, x) + b + t \left( \ovTheta_{T}(0) + B \right) \le a d_{\Gamma}(x_{0}, x) + t \left( \frac{b}{\ovt} + \ovTheta_{T}(0) + B \right),
        \end{align*}
        which proves~\zcref{eq:optvf.1}.
    \end{proof}

    \begin{prop}\label{vfloclip}
        If \(u_{0}\) is uniformly continuous then the value function \(u\) defined in~\zcref{eq:vf} is Lipschitz continuous on \(\Gamma \times \left[ \ovt, T \right]\) for each \(T > \ovt > 0\).
    \end{prop}
    \begin{proof}
        We fix \(T > \ovt > 0\). By \zcref{optvf} there is a \(C > 0\) such that, for any \((x, t) \in \Gamma \times \left[ \ovt, T \right]\), the infimum in~\zcref{eq:vf} is taken over the \(x_{0}\) with \(d_{\Gamma}(x_{0}, x) \le Ct\). It then follows that \(u\) is Lipschitz continuous on \(\Gamma \times \left[ \ovt, T \right]\) since, by virtue of \zcref{minactlip}, it is the infimum of a family of equiLipschitz continuous functions.
    \end{proof}

    \begin{rem}
        If the initial datum \(u_{0}\) is Lipschitz continuous, \zcref{optvf} holds true with \(\ovt = 0\). Arguing as in the proof of \zcref{vfloclip}, we then infer that \(u\) is Lipschitz continuous in \(\Gamma \times [0, T]\) for any \(T > 0\).
    \end{rem}

    We can now prove that the value function~\zcref{eq:vf} is continuous.

    \begin{proof}[Proof of \zcref{vfcont}]
        Fixed \(T > 0\), we assume for purpose of contradiction that there exist two sequences \(\{(x_{n}, t_{n})\}_{n \in \Nds}\) and \(\{(x_{n}', t'_{n})\}_{n \in \Nds}\) in \(\Gamma \in [0, T]\) with
        \begin{equation*}
            \lim_{n \to \infty} d_{\Gamma} \left( x_{n}, x_{n}' \right) + \left\lvert t_{n} - t_{n}' \right\rvert = 0
        \end{equation*}
        and such that, for some \(\varepsilon > 0\),
        \begin{equation}\label{eq:vfcont1}
            \left\lvert u(x_{n}, t_{n}) - u \left( x'_{n}, t_{n}' \right) \right\rvert > 3 \varepsilon \qquad \text{for every \(n \in \Nds\)}.
        \end{equation}
        Notice that by \zcref{vfloclip} both \(t_{n}\) and \(t_{n}'\) must tend to \(0\) as \(n \to \infty\). \\
        We proceed estimating
        \begin{equation*}
            \lvert u(x, t) - u_{0}(x) \rvert \qquad \text{for \((x, t) \in \Gamma \times (0, T]\)}.
        \end{equation*}
        If \(x_{0}\) is optimal for \(u(x, t)\) and \(\zeta\) is a minimizing curve for \(\Phi(x_{0}, 0, x, t)\), we have
        \begin{equation}\label{eq:vfcont2}
            u_{0}(x) - u(x, t) = u_{0}(x) - u_{0}(x_{0}) - \int_{0}^{t} L \left( \zeta, \dot{\zeta}, \tau \right) \dl \tau.
        \end{equation}
        Let \(\omega\) be a concave modulus of continuity for \(u_{0}\), then the concavity of \(\omega\) yields that there is a positive constant \(a_{\varepsilon}\) with
        \begin{equation*}
            \omega(r) \le a_{\varepsilon} r + \varepsilon \qquad \text{for any \(r \ge 0\)},
        \end{equation*}
        while \zcref{gLagprop} implies the existence of a \(B_{\varepsilon} > 0\) so that
        \begin{equation*}
            L(y, q, r) \ge \vartheta_{T}(\lvert q \rvert) \ge a_{\varepsilon} \lvert q \rvert - B_{\varepsilon} \qquad \text{for each \((y, q) \in \Trm \Gamma\) and \(r \in \Rds^{+}\)}.
        \end{equation*}
        Applying these inequalities to~\zcref{eq:vfcont2} we obtain
        \begin{equation}\label{eq:vfcont3}
            \begin{aligned}
                u_{0}(x) - u(x, t) & \le \omega(d_{\Gamma}(x_{0}, x)) - \int_{0}^{t} \left( a_{\varepsilon} \left\lvert \dot{\zeta}(\tau) \right\rvert - B_{\varepsilon} \right) \dl \tau \\
                & \le a_{\varepsilon} d_{\Gamma}(x_{0}, x) + \varepsilon - a_{\varepsilon} d_{\Gamma}(x_{0}, x) + B_{\varepsilon} t \le \varepsilon + B_{\varepsilon} t.
            \end{aligned}
        \end{equation}
        On the other hand, we have that
        \begin{equation*}
            u(x, t) - u_{0}(x) \le \int_{0}^{t} L(x, 0, \tau) \dl \tau \le t \ovTheta_{T}(0),
        \end{equation*}
        which together with~\zcref{eq:vfcont3} shows that
        \begin{equation*}
            \lvert u(x, t) - u_{0}(x) \rvert \le \varepsilon + \max \left\{ \ovTheta_{T}(0), B_{\varepsilon} \right\} t \qquad \text{for any \((x, t) \in \Gamma \times (0, T]\)}.
        \end{equation*}
        In view of this inequality,
        \begin{align*}
            \left\lvert u(x_{n}, t_{n}) - u \left( x_{n}', t_{n}' \right) \right\rvert & \le \lvert u(x_{n}, t_{n}) - u_{0}(x_{n}) \rvert + \left\lvert u_{0}(x_{n}) - u_{0} \left( x_{n}' \right) \right\rvert + \left\lvert u_{0} \left( x_{n}', t_{n}' \right) - u \left( x_{n}', t_{n}' \right) \right\rvert \\
            & \le 2 \varepsilon + \max \left\{ \ovTheta_{T}(0), B_{\varepsilon} \right\} \left( t_{n} + t_{n}' \right) + \omega \left( d_{\Gamma} \left( x_{n}, x_{n}' \right) \right),
        \end{align*}
        which contradicts~\zcref{eq:vfcont1}.
    \end{proof}

    The next \zcref[noref]{repform} is the main result of this paper.

    \begin{theo}\label{repform}
        Given a uniformly continuous \(u_{0} \colon \Gamma \to \Rds\), the function \(u\) defined in~\zcref{eq:vf} is the unique solution to~\zcref{eq:globHJ} with initial datum \(u_{0}\) and flux limiter \(c_{x}\).
    \end{theo}
    \begin{proof}
        Uniqueness is a consequence of \zcref{comp}, thus we only have to prove that \(u\) is a solution. \\
        Fixed \((x, t) \in \Gamma \times (0, \infty)\), \zcref{mincurvelip,optvf} show that there is a curve \(\zeta\) with
        \begin{equation}\label{eq:repform1}
            u(x, t) = u_{0}(\zeta(0)) + \int_{0}^{t} L \left( \zeta, \dot{\zeta}, \tau \right) \dl \tau.
        \end{equation}
        If there is a \(\gamma \in \Ecal\) and an \(s \in (0, \lvert \gamma \rvert)\) such that \(x = \gamma(s)\), then, for a \(\delta > 0\) small enough, we can define the curve \(\eta \coloneqq \gamma^{- 1} \circ \zeta |_{[t - \delta, t]}\) which is absolutely continuous by~\cite[Lemma~3.2]{PozzaSiconolfi23}. It follows from~\zcref{eq:repform1} that
        \begin{equation}\label{eq:repform2}
            u(\gamma(s), t) - u(\eta(t - \delta), t - \delta) = \int_{t - \delta}^{t} L_{\gamma}(\eta, \dot{\eta}, \tau) \dl \tau,
        \end{equation}
        therefore, since \((x, t)\) is arbitrary, \zcref[pairsep={, }]{subsolloc,supsolloc,vfcont} show that \(u(\gamma(s), t)\) is a continuous solution to~\zcref{eq:locHJ} for any \(\gamma \in \Ecal\), i.e., \(u\) satisfies \zcref[noname]{en:subsol1,en:supsol1} in \zcref{defsol}. \\
        Now assume that \(x \in \Vbf\). We have by definition that, for any \(\delta > 0\) small enough,
        \begin{equation*}
            u(x, t) - u(x, t - \delta) \le \int_{t - \delta}^{t} L(x, 0, \tau) \dl \tau = \int_{t - \delta}^{t} c_{x}(\tau) \dl \tau,
        \end{equation*}
        thus, if \(\psi\) is a \(\Crm^{1}\) supertangent to \(u(x, \cdot)\) at \(t\), \(\psi(t) \le c_{x}(t)\). This proves \zcref{en:subsol2} in \zcref{defsol}. \\
        Finally, let \(\phi\) be \(\Crm^{1}\) subtangent to \(u(x, \cdot)\) at \(t \in (0, \infty)\) and assume that \(\phi(t) < c_{x}(t)\), which implies that for any \(\delta > 0\) small enough
        \begin{equation*}
            \int_{t - \delta}^{t} c_{x}(\tau) \dl \tau > \phi(t) - \phi(t - \delta) \ge u(x, t) - u(x, t - \delta).
        \end{equation*}
        This shows that there is a \(\delta > 0\) such that \(\zeta((t - \delta, t)) \cap \Vbf = \emptyset\), i.e., there exist a \(\gamma \in \Ecal_{x}\), a \(\delta > 0\) and a curve \(\eta \colon [t - \delta, t] \to [0, \lvert \gamma \rvert]\) with \(\eta(t) = 0\) and \(\eta((t - \delta, t)) \cap \{0\} = \emptyset\) for which~\zcref{eq:repform2} holds true with \(s = 0\). Thereby, \(u\) satisfies \zcref{en:supsol2} in \zcref{defsol} by \zcref{stateconstloc}. This concludes the proof.
    \end{proof}

    \subsection{Beyond the Maximal Flux Limiter}\label{sec:beyondfl}

    In \zcref{fldef} we required flux limiters to be bounded from above by
    \begin{equation*}
        \ovc_{x}(t) \coloneqq \min_{\gamma \in \Ecal_x} L_{\gamma}(0, 0, t) \qquad \text{for \((x, t) \in \Vbf \times \Rds^{+}\)}.
    \end{equation*}
    It is natural to ask whether~\zcref{eq:globHJ} remains well-posed even for a flux limiter \(c_{x} \ge \ovc_{x}\). In this case, not only does~\zcref{eq:globHJ} admit a unique solution for a given uniformly continuous initial datum \(u_{0}\), but this solution coincides with the value function~\zcref{eq:vf} (hereafter denoted by \(u\)) corresponding to the flux limiter \(\ovc_{x}\). Indeed, straightforward modifications to the proof of \zcref{repform} show that \(u\) is a solution to~\zcref{eq:globHJ} with flux limiter \(c_{x}\) and initial datum \(u_{0}\).

    To prove uniqueness, observe that in view of \zcref{subsolloc} and \zcref{defsol}\zcref[noname]{en:subsol1}, every continuous subsolution \(w\) to~\zcref{eq:globHJ} must satisfy
    \begin{equation*}
        w(\gamma(s), t) - w(\gamma(s), t - \delta) \le \int_{t - \delta}^{t} L_{\gamma}(s, 0, \tau) \dl \tau
    \end{equation*}
    for any \(\gamma \in \Ecal\), \(s \in (0, \lvert \gamma \rvert)\), \(t > 0\) and \(\delta > 0\) small enough. The continuity of \(w\) and \(L_{\gamma}\) then yields
    \begin{equation*}
        w(x, t) - w(x, t - \delta) \le \int_{t - \delta}^{t} L_{\gamma}(0, 0, \tau) \dl \tau
    \end{equation*}
    for any \(x \in \Vbf\), \(\gamma \in \Ecal_{x}\), \(t > 0\) and \(\delta > 0\) small enough. Consequently, for each \(x \in \Vbf\), \(t > 0\) and \(\Crm^{1}\) supertangent \(\psi\) to \(w(x, \cdot)\) at \(t\)
    \begin{equation*}
        \difcp{\psi}{}(t) \le \ovc_{x}(t).
    \end{equation*}
    In other words, any subsolution to~\zcref{eq:globHJ} with flux limiter \(c_{x}\) is also a subsolution with flux limiter \(\ovc_{x}\). Additionally, it is apparent from \zcref{defsol}\zcref[noname]{en:supsol1,en:supsol2} that any supersolution to~\zcref{eq:globHJ} with flux limiter \(c_{x}\) is also a supersolution with flux limiter \(\ovc_{x}\). Since \(u\) is the unique solution to~\zcref{eq:globHJ} with flux limiter \(\ovc_{x}\) and initial datum \(u_{0}\), it is also the only solution with flux limiter \(c_{x}\) and initial datum \(u_{0}\).

    Furthermore, the same arguments show that, given any flux limiter \(c_{x}\) such that
    \begin{equation*}
        c'_{x}(t) \coloneqq \min \{c_{x}(t), \ovc_{x}(t)\}
    \end{equation*}
    satisfies \zcref{condflbound,condfllip}, the value function~\zcref{eq:vf} corresponding to the flux limiter \(c'_{x}\) is the unique solution to~\zcref{eq:globHJ} with flux limiter \(c_{x}\) and initial datum \(u_{0}\).

    \section{Optimal Curves of the Minimal Action}\label{sec:mincurve}

    This \zcref[noref,nocap]{sec:mincurve} is devoted to the proofs of \zcref{mincurvelip,minactlip}.

    Let \(\difcp[-]{L(x, q, t)}{q}\) denote the subdifferential of \(L\) in \(q\), which is always nonempty thanks to the convexity of \(L\). We recall that if \(p \in \difcp[-]{L(x, q, t)}{q}\) then
    \begin{equation*}
        L^{*}(x, p, t) = q \cdot p - L(x, q, t),
    \end{equation*}
    where \(L^{*}\) is the convex conjugate of \(L\) with respect to the speed variable. We further set \(\vartheta_{T}^{*}\) and \(\Theta_{T}^{*}\) as the convex conjugates of \(\ovTheta_{T}\) and \(\ovvartheta_{T}\), respectively, defined in \zcref{gLagprop}. It is easy to check that
    \begin{equation}\label{eq:boundgLagdual}
        \vartheta_{T}^{*}(\lvert p \rvert) \le L^{*}(x, p, t) \le \Theta_{T}^{*}(\lvert p \rvert).
    \end{equation}

    \begin{lem}\label{dualbound}
        Given two constant \(A\) and \(C \ge 0\), there is an \(R > 0\) such that
        \begin{multline*}
            \sup \left\{ \Theta^{*}_{T}(\lvert p \rvert) : p \in \difcp[-]{L(x, q, t)}{q}, \, (x, q) \in \Trm \Gamma, \, t \in [0, T], \, \lvert q \rvert \le C \right\} + A \\
            \le \inf \left\{ \vartheta^{*}_{T}(\lvert p \rvert) : p \in \difcp[-]{L(x, q, t)}{q}, \, (x, q) \in \Trm \Gamma, \, t \in [0, T], \, \lvert q \rvert > R \right\}.
        \end{multline*}
    \end{lem}
    \begin{proof}
        It is known, see for instance~\cite[Lemma~1 to Theorem~2.34]{Clarke13}, that, for any \((x, q) \in \Trm \Gamma\) and \(t \in [0, T]\) such that \(\lvert q \rvert \le C\),
        \begin{equation}\label{eq:dualbound1}
            \lvert p \rvert \le \ovTheta_{T}(C + 2) \qquad \text{for every \(p \in \difcp[-]{L(x, q, t)}{q}\)}.
        \end{equation}
        Furthermore, if \(p' \in \difcp[-]{L(x, q, t)}{q}\), the subgradient inequality yields
        \begin{equation*}
            L(x, 0, t) - L(x, q, t) \ge - q \cdot p',
        \end{equation*}
        from which follows
        \begin{equation*}
            \ovvartheta_{T}(\lvert q \rvert) - \ovTheta_{T}(0) \le L(x, q, t) - L(x, 0, t) \le \lvert q \rvert \lvert p' \rvert.
        \end{equation*}
        Since \(\ovvartheta_{T}\) is superlinearly coercive, this shows that for any \(p' \in \difcp[-]{L(x, q, t)}{q}\)
        \begin{equation}\label{eq:dualbound2}
            \lvert p' \rvert \longrightarrow \infty \qquad \text{as} \qquad \lvert q \rvert \longrightarrow \infty.
        \end{equation}
        \(\vartheta_{T}^{*}\) is coercive as the convex conjugate of a coercive function, thereby~\zcref{eq:dualbound1,eq:dualbound2} imply the claim.
    \end{proof}

    \begin{prop}\label{minseqminact}
        Fixed \(T > 0\) and \(C > 0\), there is a positive constant \(\kappa\) such that, for every \(x_{1}, x_{2} \in \Gamma\) and \(t_{1}, t_{2} \in [0, T]\) satisfying
        \begin{equation*}
            d_{\Gamma}(x_{1}, x_{2}) \le C(t_{2} - t_{1}),
        \end{equation*}
        there exists a minimizing sequence for \(\Phi(x_{1}, t_{1}, x_{2}, t_{2})\) consisting of \(\kappa\)-Lipschitz continuous curves.
    \end{prop}
    \begin{proof}
        By~\cite[Proposition~2.5.9 and Corollary~2.5.20]{BuragoBuragoIvanov01} there exists a geodesic \(\zeta \colon [t_{1}, t_{2}] \to \Gamma\) linking \(x_{1}\) to \(x_{2}\) with
        \begin{equation*}
            \left\lvert \dot{\zeta}(t) \right\rvert = \frac{d_{\Gamma}(x_{1}, x_{2})}{t_{2} - t_{1}} \le C \qquad \text{for a.e.\ \(t \in [t_{1}, t_{2}]\)},
        \end{equation*}
        then
        \begin{equation}\label{eq:minseqminact1}
            \int_{t_{1}}^{t_{2}} L \left( \zeta, \dot{\zeta}, \tau \right) \dl \tau \le \int_{t_{1}}^{t_{2}} \ovTheta_{T} \left( \left\lvert \dot{\zeta}(\tau) \right\rvert \right) \dl \tau \le (t_{2} - t_{1}) \ovTheta_{T}(C).
        \end{equation}
        Furthermore, the superlinearity of \(\ovvartheta_{T}\) implies the existence of a constant \(A > 0\) so that
        \begin{equation*}
            \ovvartheta_{T}(r) \ge r - A \qquad \text{for all \(r \in \Rds^{+}\)},
        \end{equation*}
        hence, for any curve \(\xi \colon [t_{1}, t_{2}] \to \Gamma\),
        \begin{equation}\label{eq:minseqminact2}
            \int_{t_{1}}^{t_{2}} L \left( \xi, \dot{\xi}, \tau \right) \dl \tau \ge \int_{t_{1}}^{t_{2}} \ovvartheta_{T} \left( \left\lvert \dot{\xi}(\tau) \right\rvert \right) \dl \tau \ge \int_{t_{1}}^{t_{2}} \left\lvert \dot{\xi}(\tau) \right\rvert \dl \tau - (t_{2} - t_{1}) A.
        \end{equation}
        We can assume that for each curve \(\xi\) of a minimizing sequence
        \begin{equation*}
            \int_{t_{1}}^{t_{2}} L \left( \xi, \dot{\xi}, \tau \right) \dl \tau \le \int_{t_{1}}^{t_{2}} L \left( \zeta, \dot{\zeta}, \tau \right) \dl \tau,
        \end{equation*}
        consequently, setting \(M \coloneqq A + \ovTheta_{T}(C)\), \zcref{eq:minseqminact1,eq:minseqminact2} yield
        \begin{equation}\label{eq:minseqminact4}
            \int_{t_{1}}^{t_{2}} \left\lvert \dot{\xi}(\tau) \right\rvert \dl \tau \le (t_{2} - t_{1}) M.
        \end{equation}
        The claim will then follow once we show the existence of a constant \(\kappa\) such that, for any given curve \(\xi \colon [t_{1}, t_{2}] \to \Gamma\) satisfying~\zcref{eq:minseqminact4}, there is a \(\kappa\)-Lipschitz continuous curve \(\ovxi \colon [t_{1}, t_{2}] \to \Gamma\) with the same support and such that
        \begin{equation}\label{eq:minseqminact5}
            \int_{t_{1}}^{t_{2}} L \left( \ovxi, \dot{\ovxi}, \tau \right) \dl \tau \le \int_{t_{1}}^{t_{2}} L \left( \xi, \dot{\xi}, \tau \right) \dl \tau.
        \end{equation}
        We break the proof in six steps.
        \begin{steproof}
            \item \emph{Choice of \(\kappa\) and definitions of \(\Lambda\), \(\Xi\) and \(\Upsilon\).} \\
            Taking into account \zcref{dualbound}, we select a constant \(\kappa\) such that
            \begin{equation}\label{eq:minseqminact6}
                \kappa > \frac{4MT}{\ovepsilon_{T}} \qquad \text{and} \qquad \Lambda + \Xi \le \Upsilon(\kappa),
            \end{equation}
            where
            \begin{gather*}
                \Lambda \coloneqq 2 \left( \ovalpha_{T} M + \ovbeta_{T} \right) T, \\
                \Xi \coloneqq \sup \left\{ \Theta^{*}_{T}(\lvert p \rvert) : p \in \difcp[-]{L(x, q, t)}{q}, \, (x, q) \in \Trm \Gamma, \, t \in [0, T], \, \lvert q \rvert \le 4M \right\}
            \end{gather*}
            and
            \begin{equation*}
                \Upsilon(\kappa) \coloneqq \inf \left\{ \vartheta^{*}_{T}(\lvert p \rvert) : p \in \difcp[-]{L(x, q, t)}{q}, \, (x, q) \in \Trm \Gamma, \, t \in [0, T], \, \lvert q \rvert > \kappa \right\}.
            \end{equation*}
            \item \emph{Definitions of \(E_{\kappa}\), \(F_{\kappa}\) and \(\varepsilon_{\kappa}\).} \\
            We define
            \begin{equation*}
                E_{\kappa} \coloneqq \left\{ t \in [t_{1}, t_{2}] : \left\lvert \dot{\xi}(t) \right\rvert > \kappa \right\}
            \end{equation*}
            and
            \begin{equation*}
                \varepsilon_{\kappa} \coloneqq \int_{E_{\kappa}} \left( \frac{1}{\kappa} \left\lvert \dot{\xi}(\tau) \right\rvert - 1 \right) \dl \tau \ge 0.
            \end{equation*}
            Notice that by~\zcref{eq:minseqminact4,eq:minseqminact6}
            \begin{equation*}
                \varepsilon_{\kappa} \le \frac{1}{\kappa} \int_{t_{1}}^{t_{2}} \left\lvert \dot{\xi}(\tau) \right\rvert \dl \tau \le \frac{(t_{2} - t_{1}) M} \kappa < \frac{(t_{2} - t_{1}) \ovepsilon_{T}}{4T},
            \end{equation*}
            therefore
            \begin{equation}\label{eq:minseqminact7}
                \varepsilon_{\kappa} < \frac{1}{4} \min \{\ovepsilon_{T}, t_{2} - t_{1}\}.
            \end{equation}
            Furthermore, it follows from~\zcref{eq:minseqminact4} that
            \begin{equation*}
                \left\lvert \left\{ t \in [t_{1}, t_{2}] : \left\lvert \dot{\xi}(t) \right\rvert \le 2M \right\} \right\rvert \ge \frac{t_{2} - t_{1}}{2};
            \end{equation*}
            we can thus choose a set \(F_{\kappa} \subseteq \left\{ t \in [t_{1}, t_{2}] : \left\lvert \dot{\xi}(t) \right\rvert \le 2M \right\}\) such that
            \begin{equation*}
                \lvert F_{\kappa} \rvert = 2 \varepsilon_{\kappa} < \frac{t_{2} - t_{1}}{2}.
            \end{equation*}
            We know from~\zcref{eq:minseqminact6} that \(\kappa > 2M\), thereby \(E_{\kappa}\) and \(F_{\kappa}\) are disjointed.
            \item \emph{Definitions of \(\ovxi\) and \(\varphi\).} \\
            We set the absolutely continuous function \(\varphi \colon [t_{1}, t_{2}] \to \Rds\) defined by
            \begin{equation*}
                \varphi(t_{1}) = t_{1} \qquad \text{and} \qquad \dot{\varphi}(t) =
                \begin{dcases}
                    \frac{1}{\kappa} \left\lvert \dot{\xi}(t) \right\rvert & \text{if \(t \in E_{\kappa}\)}, \\
                    \frac{1}{2} & \text{if \(t \in F_{\kappa}\)}, \\
                    1 & \text{otherwise}.
                \end{dcases}
            \end{equation*}
            The map \(\varphi\) is increasing and
            \begin{align*}
                \varphi(t_{2}) - \varphi(t_{1}) & = \int_{t_{1}}^{t_{2}} \dot{\varphi}(\tau) \dl \tau = \int_{E_{\kappa}} \frac{1}{\kappa} \left\lvert \dot{\xi}(\tau) \right\rvert \dl \tau + \frac{1}{2} \lvert F_{\kappa} \rvert + t_{2} - t_{1} - \lvert E_{\kappa} \rvert - \lvert F_{\kappa} \rvert \\
                & = \int_{E_{\kappa}} \left( \frac{1}{\kappa} \left\lvert \dot{\xi}(\tau) \right\rvert - 1 \right) \dl \tau - \frac{1}{2} \lvert F_{\kappa} \rvert + t_{2} - t_{1} = \varepsilon_{\kappa} - \varepsilon_{\kappa} + t_{2} - t_{1} = t_{2} - t_{1},
            \end{align*}
            therefore \(\varphi\) is a bijective map from \([t_{1}, t_{2}]\) into itself and admits an increasing absolutely continuous inverse, denoted by \(\psi\). By virtue of~\cite[Corollary 4]{SerrinVarberg69}, \(\ovxi \coloneqq \xi \circ \psi\) is \(\kappa\)-Lipschitz continuous:
            \begin{equation*}
                \dot{\ovxi}(t) = \frac{\dot{\xi}(\psi(t))}{\dot{\varphi}(\psi(t))} \le \kappa.
            \end{equation*}
            \item \emph{We have that}
            \begin{equation}\label{eq:minseqminact8}
                \lvert \varphi(t) - t \rvert \le 2 \varepsilon_{\kappa} < \ovepsilon_{T} \qquad \text{\emph{for any \(t \in [t_{1}, t_{2}]\)}}.
            \end{equation}
            Indeed, since \(\varphi(t_{1}) = t_{1}\), \zcref{eq:minseqminact7} yields that for any \(t \in [t_{1}, t_{2}]\)
            \begin{equation*}
                \lvert \varphi(t) - t \rvert \le \int_{t_{1}}^{t_{2}} \lvert \dot{\varphi}(\tau) - 1 \rvert \dl \tau = \int_{E_{\kappa}} \left( \frac{1}{\kappa} \left\lvert \dot{\xi}(\tau) \right\rvert - 1 \right) \dl \tau + \int_{F_{\kappa}} \left( 1 - \frac{1}{2} \right) \dl \tau = 2 \varepsilon_{\kappa} < \ovepsilon_{T}.
            \end{equation*}
            \item \emph{We prove that}
            \begin{equation}\label{eq:minseqminact9}
                \int_{t_{1}}^{t_{2}} L \left( \ovxi, \dot{\ovxi}, \tau \right) \dl \tau \le \int_{t_{1}}^{t_{2}} L \left( \xi(\tau), \dot{\xi}(\tau), \varphi(\tau) \right) \dl \tau + \varepsilon_{\kappa}(\Xi - \Upsilon(\kappa)).
            \end{equation}
            We use the change of variables \(\tau = \varphi(r)\) to obtain
            \begin{equation}\label{eq:minseqminact10}
                \begin{aligned}
                    \int_{t_{1}}^{t_{2}} L \left( \ovxi, \dot{\ovxi}, \tau \right) \dl \tau & = \int_{t_{1}}^{t_{2}} L \left( \xi(r), \frac{\dot{\xi}(r)}{\dot{\varphi}(r)}, \varphi(r) \right) \dot{\varphi}(r) \dl r \\
                    & = \int_{E_{\kappa}} L \left( \xi(r), \kappa \frac{\dot{\xi}(r)}{\left\lvert \dot{\xi}(r) \right\rvert}, \varphi(r) \right) \frac{\left\lvert \dot{\xi}(r) \right\rvert}{\kappa} \dl r + \frac{1}{2} \int_{F_{\kappa}} L \left( \xi(r), 2 \dot{\xi}(r), \varphi(r) \right) \dl r \\
                    & \mathrel{\hphantom{=}} + \int_{[t_{1}, t_{2}] \setminus(E_{\kappa} \cup F_{\kappa})} L \left( \xi(r), \dot{\xi}(r), \varphi(r) \right) \dl r.
                \end{aligned}
            \end{equation}
            Notice that if \(p \in \difcp[-]{L \left( x, \dfrac{q}{a}, t \right)}{q}\) for some \(a > 0\), then
            \begin{equation*}
                (1 - a) L^{*}(x, p, t) \ge p \cdot q - L(x, q, t) - a \left( p \cdot \frac{q}{a} - L \left( x, \frac{q}{a}, t \right) \right) \ge a L \left( x, \frac{q}{a}, t \right) - L(x, q, t).
            \end{equation*}
            Accordingly, taking into account~\zcref{eq:boundgLagdual},
            \begin{equation*}
                L \left( \xi(t), \kappa \frac{\dot{\xi}(t)}{\left\lvert \dot{\xi}(t) \right\rvert}, \varphi(t) \right) \frac{\left\lvert \dot{\xi}(t) \right\rvert}{\kappa} - L \left( \xi(t), \dot{\xi}(t), \varphi(t) \right) \le \left( 1 - \frac{\left\lvert \dot{\xi}(t) \right\rvert}{\kappa} \right) \Upsilon(\kappa),
            \end{equation*}
            which in turn yields
            \begin{equation}\label{eq:minseqminact11}
                \int_{E_{\kappa}} L \left( \xi(r), \kappa \frac{\dot{\xi}(r)}{\left\lvert \dot{\xi}(r) \right\rvert}, \varphi(r) \right) \frac{\left\lvert \dot{\xi}(r) \right\rvert}{\kappa} \dl r \le \int_{E_{\kappa}} L \left( \xi(r), \dot{\xi}(r), \varphi(r) \right) \dl r - \varepsilon_{\kappa} \Upsilon(\kappa).
            \end{equation}
            Similarly, we get
            \begin{equation}\label{eq:minseqminact12}
                \frac{1}{2} \int_{F_{\kappa}} L \left( \xi(r), 2 \dot{\xi}(r), \varphi(r) \right) \dl r \le \int_{F_{\kappa}} L \left( \xi(r), \dot{\xi}(r), \varphi(r) \right) \dl r + \varepsilon_{\kappa} \Xi.
            \end{equation}
            Combining~\zcref{eq:minseqminact12,eq:minseqminact10,eq:minseqminact11} we get~\zcref{eq:minseqminact9}.
            \item \emph{Conclusion.} \\
            Taking into account~\zcref{eq:minseqminact8,gLagprop}, we have that
            \begin{equation*}
                L \left( \xi(t), \dot{\xi}(t), \varphi(t) \right) \le L \left( \xi, \dot{\xi}, t \right) + 2 \varepsilon_{\kappa} \left( \ovalpha_{T} \left\lvert \dot{\xi}(t) \right\rvert + \ovbeta_{T} \right),
            \end{equation*}
            thereby, thanks to~\zcref{eq:minseqminact4},
            \begin{align*}
                \int_{t_{1}}^{t_{2}} L \left( \xi(\tau), \dot{\xi}(\tau), \varphi(\tau) \right) \dl \tau & \le \int_{t_{1}}^{t_{2}} L \left( \xi, \dot{\xi}, \tau \right) \dl \tau + 2 \varepsilon_{\kappa} \left( \ovalpha_{T} M + \ovbeta_{T} \right) (t_{2} - t_{1}) \\
                & \le \int_{t_{1}}^{t_{2}} L \left( \xi, \dot{\xi}, \tau \right) \dl \tau + \varepsilon_{\kappa} \Lambda.
            \end{align*}
            Applying this inequality to~\zcref{eq:minseqminact9} we get
            \begin{equation*}
                \int_{t_{1}}^{t_{2}} L \left( \ovxi, \dot{\ovxi}, \tau \right) \dl \tau \le \int_{t_{1}}^{t_{2}} L \left( \xi, \dot{\xi}, \tau \right) \dl \tau + \varepsilon_{\kappa}(\Lambda + \Xi - \Upsilon(\kappa)),
            \end{equation*}
            which in turn, together with~\zcref{eq:minseqminact6}, implies~\zcref{eq:minseqminact5}.\qedhere
        \end{steproof}
    \end{proof}

    We can now prove the main statement of this \zcref[noref,nocap]{sec:mincurve}.

    \mincurvelip*
    \begin{proof}
        Given \(x_{1}, t_{1}, x_{2}, t_{2}\) as in the statement, let \(\{\xi_{n}\}_{n \in \Nds}\) be a minimizing sequence of curves for \(\Phi(x_{1}, t_{1}, x_{2}, t_{2})\). \zcref[S]{minseqminact} yields the existence of a constant \(\kappa\) such that each \(\xi_n\) can be taken to be \(\kappa\)-Lipschitz continuous. Since the \(\xi_{n}\) are equiLipschitz continuous curves with fixed endpoints, they are also equibounded; Arzelà--Ascoli's Theorem then implies that, up to subsequences, they converge uniformly to a \(\kappa\)-Lipschitz continuous curve \(\zeta\) on \([t_{1}, t_{2}]\) linking \(x_{1}\) to \(x_{2}\). Furthermore, the \(\dot{\xi}_{n}\) are equibounded and thus, up to subsequences, they weakly converge to some \(\eta \colon [t_{1}, t_{2}] \to \Rds^{N}\) in \(\Lrm^{2}\). We have that \(\eta = \dot{\zeta}\) because, for each \(\varphi \in \Crm^{\infty}([t_{1}, t_{2}])\),
        \begin{equation*}
            \int_{t_{1}}^{t_{2}} \zeta(\tau) \cdot \dot{\varphi}(\tau) \dl \tau = \lim_{n \to \infty} \int_{t_{1}}^{t_{2}} \xi_{n}(\tau) \cdot \dot{\varphi}(\tau) \dl \tau = - \lim_{n \to \infty} \int_{t_{1}}^{t_{2}} \dot{\xi}_{n}(\tau) \cdot \varphi(\tau) \dl \tau = - \int_{t_{1}}^{t_{2}} \eta(\tau) \cdot \varphi(\tau) \dl \tau.
        \end{equation*}
        Finally, \zcref{actlsc} ensures that \(\zeta\) is the sought minimizer.
    \end{proof}

    The next \zcref[noref]{minactlip} is a consequence of \zcref{mincurvelip}.

    \minactlip*
    \begin{proof}
        We proceed by steps.
        \begin{steproof}
            \item\label{step:minactlip1} \emph{If \((x_{1}, t_{1}, x_{2}, t_{2}), (x_{1}, t_{1}, x_{2}, t_{2}') \in A_{C}\) with \(t_{2}' > t_{2}\), there is a constant \(\ell_{T}\) such that}
            \begin{equation}\label{eq:minactlip1}
                \left\lvert \Phi(x_{1}, t_{1}, x_{2}, t_{2}) - \Phi \left( x_{1}, t_{1}, x_{2}, t_{2}' \right) \right\rvert \le \ell_{T} \left\lvert t_{2}' - t_{2} \right\rvert.
            \end{equation}
            We first assume that \(\Phi(x_{1}, t_{1}, x_{2}, t_{2}') \ge \Phi(x_{1}, t_{1}, x_{2}, t_{2})\), then
            \begin{equation}\label{eq:minactlip2}
                \begin{aligned}
                    \left\lvert \Phi(x_{1}, t_{1}, x_{2}, t_{2}) - \Phi \left( x_{1}, t_{1}, x_{2}, t_{2}' \right) \right\rvert & = \Phi \left( x_{1}, t_{1}, x_{2}, t_{2}' \right) - \Phi(x_{1}, t_{1}, x_{2}, t_{2}) \\
                    & \le \Phi(x_{1}, t_{1}, x_{2}, t_{2}) + \Phi \left( x_{2}, t_{2}, x_{2}, t_{2}' \right) - \Phi(x_{1}, t_{1}, x_{2}, t_{2}) \\
                    & \le \int_{t_{2}}^{t_{2}'} L(x_{2}, 0, \tau) \dl \tau \le \left\lvert t_{2}' - t_{2} \right\rvert \ovTheta_{T}(0).
                \end{aligned}
            \end{equation}
            If instead \(\Phi(x_{1}, t_{1}, x_{2}, t_{2}') < \Phi(x_{1}, t_{1}, x_{2}, t_{2})\), we break the argument according to whether \(t_{2}' - t_{2} \ge t_{2} - t_{1}\) or \(t_{2}' - t_{2} < t_{2} - t_{1}\). \zcref[S]{mincurvelip} yields the existence of a constant \(\kappa_{T}\) so that \(\Phi\) admits a \(\kappa_{T}\)-Lipschitz continuous optimal curve on \(A_{C}\), thus, in the first case, we deduce from \zcref{gLagprop} that
            \begin{equation}\label{eq:minactlip3}
                \begin{aligned}
                    \left\lvert \Phi(x_{1}, t_{1}, x_{2}, t_{2}) - \Phi \left( x_{1}, t_{1}, x_{2}, t_{2}' \right) \right\rvert & = \Phi(x_{1}, t_{1}, x_{2}, t_{2}) - \Phi \left( x_{1}, t_{1}, x_{2}, t_{2}' \right) \\
                    & \le \int_{t_{1}}^{t_{2}} \ovTheta_{T}(\kappa_{T}) \dl \tau - \int_{t_{1}}^{t_{2}'} \ovvartheta_{T}(0) \dl \tau \\
                    & = (t_{2} - t_{1}) \left( \ovTheta_{T}(\kappa_{T}) - \ovvartheta_{T}(0) \right) - \left( t_{2}' - t_{2} \right) \ovvartheta_{T}(0) \\
                    & \le \left( t_{2}' - t_{2} \right) \left( \ovTheta_{T}(\kappa_{T}) - 2 \ovvartheta_{T}(0) \right).
                \end{aligned}
            \end{equation}
            Now assume that \(t_{2}' - t_{2} < t_{2} - t_{1}\) and let \(\zeta\) be a \(\kappa_{T}\)-Lipschitz continuous optimal curve for \(\Phi(x_{1}, t_{1}, x_{2}, t_{2}')\). We further define \(\ovt \coloneqq 2t_{2} - t_{2}'\) and
            \begin{equation*}
                \xi(t) \coloneqq
                \begin{dcases}
                    \zeta(t) & \text{if \(t \in \left[ t_{1}, \ovt \right]\)}, \\
                    \zeta \left( 2t - \ovt \right) & \text{if \(t \in \left( \ovt, t_{2} \right]\)},
                \end{dcases}
            \end{equation*}
            which is a \(2 \kappa_{T}\)-Lipschitz continuous curve. Consequently,
            \begin{multline}\label{eq:minactlip4}
                \left\lvert \Phi(x_{1}, t_{1}, x_{2}, t_{2}) - \Phi \left( x_{1}, t_{1}, x_{2}, t_{2}' \right) \right\rvert \\
                \begin{aligned}
                    & = \Phi(x_{1}, t_{1}, x_{2}, t_{2}) - \Phi \left( x_{1}, t_{1}, x_{2}, t_{2}' \right) \le \int_{\ovt}^{t_{2}} L \left( \xi, \dot{\xi}, \tau \right) \dl \tau - \int_{\ovt}^{t_{2}'} L \left( \zeta, \dot{\zeta}, \tau \right) \dl \tau \\
                    & \le \int_{\ovt}^{t_{2}} \ovTheta_{T}(2 \kappa_{T}) \dl \tau - \int_{\ovt}^{t_{2}'} \ovvartheta_{T}(0) \dl \tau = \left( t_{2} - \ovt \right) \ovTheta_{T}(2 \kappa_{T}) - \left( t_{2}' - \ovt \right) \ovvartheta_{T}(0) \\
                    & = \left( t_{2}' - t_{2} \right) \left( \ovTheta_{T}(2 \kappa_{T}) - 2 \ovvartheta_{T}(0) \right).
                \end{aligned}
            \end{multline}
            Setting \(\ell_{T} \coloneqq \max \left\{ \ovTheta_{T}(0), \ovTheta_{T}(2 \kappa_{T}) - 2 \ovvartheta_{T}(0) \right\}\), \zcref{eq:minactlip2,eq:minactlip3,eq:minactlip4} prove \zcref{eq:minactlip1}.
            \item \emph{The same arguments used in \zcref{step:minactlip1} show that, for any \((x_{1}, t_{1}, x_{2}, t_{2}), (x_{1}, t_{1}', x_{2}, t_{2}') \in A_{C}\),}
            \begin{equation}\label{eq:minactlip5}
                \left\lvert \Phi(x_{1}, t_{1}, x_{2}, t_{2}) - \Phi \left( x_{1}, t_{1}', x_{2}, t_{2}' \right) \right\rvert \le \ell_{T} \left( \left\lvert t_{2}' - t_{2} \right\rvert + \left\lvert t_{1}' - t_{1} \right\rvert \right).
            \end{equation}
            \item \emph{There is an \(\ovell \in \Rds\) such that, for any \((x_{1}, t_{1}, x_{2}, t_{2}), (x_{1}', t_{1}, x_{2}', t_{2}) \in A_{C}\),}
            \begin{equation}\label{eq:minactlip6}
                \left\lvert \Phi(x_{1}, t_{1}, x_{2}, t_{2}) - \Phi \left( x_{1}', t_{1}, x_{2}', t_{2} \right) \right\rvert \le \ovell \left( d_{\Gamma} \left( x_{2}, x_{2}' \right) + d_{\Gamma} \left( x_{1}, x_{1}' \right) \right).
            \end{equation}
            \begin{steproof}
                \item\label{step:minactlip3a} \emph{\zcref[S]{eq:minactlip6} holds true whenever}
                \begin{equation}\label{eq:minactlip7}
                    d_{\Gamma} \left( x_{1}, x_{1}' \right) < C \min \{T, \ovepsilon_{2T}\} \qquad \text{\emph{and}} \qquad d_{\Gamma} \left( x_{2}, x_{2}' \right) \le CT.
                \end{equation}
                We define
                \begin{equation}\label{eq:minactlip8}
                    T_{1} \coloneqq \frac{d_{\Gamma} \left( x_{1}, x_{1}' \right)}{C}, \qquad T_{2} \coloneqq \frac{d_{\Gamma} \left( x_{2}, x_{2}' \right)}{C}.
                \end{equation}
                Clearly, \(d_{\Gamma}(x'_{1}, x'_{2}) \le C (T_{1} + T_{2} + t_{2} - t_{1})\) and, according to~\zcref{eq:minactlip7}, \(t_{2} + T_{1} + T_{2} \le 3T\); therefore~\zcref{eq:minactlip5} yields
                \begin{multline}\label{eq:minactlip9}
                    \Phi \left( x_{1}', t_{1}, x_{2}', t_{2} \right) - \Phi(x_{1}, t_{1}, x_{2}, t_{2}) \\
                    \begin{aligned}
                        & \le \Phi \left( x_{1}', t_{1}, x_{2}', t_{2} \right) - \Phi \left( x'_{1}, t_{1}, x'_{2}, t_{2} + T_{1} + T_{2} \right) \\
                        & \mathrel{\hphantom{\le}} + \Phi \left( x_{1}', t_{1}, x_{2}', t_{2} + T_{1} + T_{2} \right) - \Phi(x_{1}, t_{1}, x_{2}, t_{2}) \\
                        & \le \Phi \left( x_{1}', t_{1}, x_{2}', t_{2} + T_{1} + T_{2} \right) - \Phi(x_{1}, t_{1}, x_{2}, t_{2}) + \ell_{3T}(T_{1} + T_{2}) \\
                        & = \Phi \left( x_{1}', t_{1}, x_{2}', t_{2} + T_{1} + T_{2} \right) - \Phi(x_{1}, t_{1}, x_{2}, t_{2}) + \frac{\ell_{3T}}{C} \left( d_{\Gamma} \left( x_{1}, x_{1}' \right) + d_{\Gamma} \left( x_{2}, x_{2}' \right) \right).
                    \end{aligned}
                \end{multline}
                Furthermore, in view of~\zcref[tlastsep={ and }]{eq:minactlip8,mincurvelip,gLagprop},
                \begin{multline}\label{eq:minactlip10}
                    \Phi \left( x_{1}', t_{1}, x_{2}', t_{2} + T_{1} + T_{2} \right) \\
                    \begin{aligned}
                        & \le \Phi \left( x_{1}', t_{1}, x_{1}, t_{1} + T_{1} \right) + \Phi(x_{1}, t_{1} + T_{1}, x_{2}, t_{2} + T_{1}) + \Phi \left( x_{2}, t_{2} + T_{1}, x_{2}', t_{2} + T_{1} + T_{2} \right) \\
                        & \le \Phi(x_{1}, t_{1} + T_{1}, x_{2}, t_{2} + T_{1}) + (T_{1} + T_{2}) \ovTheta_{3T}(\kappa_{3T}) \\
                        & \le \Phi(x_{1}, t_{1} + T_{1}, x_{2}, t_{2} + T_{1}) + \frac{\ovTheta_{3T}(\kappa_{3T})}{C} \left( d_{\Gamma} \left( x_{1}, x_{1}' \right) + d_{\Gamma} \left( x_{2}, x_{2}' \right) \right).
                    \end{aligned}
                \end{multline}
                To conclude, let \(\zeta\) be the optimal curve for \(\Phi(x_{1}, t_{1}, x_{2}, t_{2})\); then
                \begin{multline}\label{eq:minactlip11}
                    \Phi(x_{1}, t_{1} + T_{1}, x_{2}, t_{2} + T_{1}) - \Phi(x_{1}, t_{1}, x_{2}, t_{2}) \\
                    \begin{aligned}
                        & \le \int_{t_{1} + T_{1}}^{t_{2} + T_{1}} L \left( \zeta(\tau - T_{1}), \dot{\zeta}(\tau - T_{1}), \tau \right) \dl \tau - \int_{t_{1}}^{t_{2}} L \left( \zeta, \dot{\zeta}, \tau \right) \dl \tau \\
                        & = \int_{t_{1}}^{t_{2}} L \left( \zeta, \dot{\zeta}, r + T_{1} \right) \dl r - \int_{t_{1}}^{t_{2}} L \left( \zeta, \dot{\zeta}, \tau \right) \dl \tau,
                    \end{aligned}
                \end{multline}
                where the last identity is due to the change of variable \(r = \tau - T_{1}\). Since \(T_{1} < \ovepsilon_{2T}\) we further have
                \begin{equation}\label{eq:minactlip12}
                    \begin{aligned}
                        \int_{t_{1}}^{t_{2}} L \left( \zeta, \dot{\zeta}, r + T_{1} \right) \dl r - \int_{t_{1}}^{t_{2}} L \left( \zeta, \dot{\zeta}, \tau \right) \dl \tau & \le \left( \ovalpha_{2T} \int_{t_{1}}^{t_{2}} \left\lvert \dot{\zeta}(\tau) \right\rvert \dl \tau + \ovbeta_{2T}(t_{2} - t_{1}) \right) T_{1} \\
                        & \le T \left( \ovalpha_{2T} \kappa_{T} + \ovbeta_{2T} \right) \frac{d_{\Gamma} \left( x_{1}, x_{1}' \right)}{C}.
                    \end{aligned}
                \end{equation}
                Combining \zcref{eq:minactlip12,eq:minactlip11,eq:minactlip9,eq:minactlip10} and interchanging the roles of \((x_{1}, t_{1}, x_{2}, t_{2})\) and \((x_{1}', t_{1}, x_{2}', t_{2})\) completes the proof of this step.
                \item\label{step:minactlip3b} \emph{If \(x_{1}' = x_{1}\), \zcref{eq:minactlip6} holds true.} \\
                \zcref[S]{step:minactlip3a} proves the claim when \(d_{\Gamma}(x_{2}, x_{2}') \le CT\), thus we assume that
                \begin{equation*}
                    d_{\Gamma} \left( x_{2}, x_{2}' \right) > CT \ge C(t_{2} - t_{1}).
                \end{equation*}
                Exploiting \zcref{gLagprop,mincurvelip}, we obtain that
                \begin{equation*}
                    \left\lvert \Phi(x_{1}, t_{1}, x_{2}, t_{2}) - \Phi \left( x_{1}, t_{1}, x_{2}', t_{2} \right) \right\rvert \le (t_{2} - t_{1}) \left( \ovTheta_{T}(\kappa_{T}) - \ovvartheta_{T}(0) \right) \le \frac{\ovTheta_{T}(\kappa_{T}) - \ovvartheta_{T}(0)}{C} d_{\Gamma} \left( x_{2}, x_{2}' \right),
                \end{equation*}
                which concludes this step.
                \item \emph{Proof of~\zcref{eq:minactlip6}.} \\
                We fix two finite sequences \(\{z_{1, i}\}_{i = 1}^{m}\) and \(\{z_{2, i}\}_{i = 1}^{m}\) such that
                \begin{mylist}
                    \item \(z_{1, 1} = x_{1}\), \(z_{1, m} = x_{1}'\) and \(z_{2, 1} = x_{2}\);
                    \item \((z_{1, i}, t_{1}, z_{2, i}, t_{2}) \in A_{C}\) for all \(i \in \{1, \dotsc, m\}\);
                    \item \(d_{\Gamma}(z_{2, i}, z_{2, i + 1}) \le d_{\Gamma}(z_{1, i}, z_{1, i + 1}) < C \min \{T, \epsilon_{2T}\}\) for any \(i \in \{1, \dotsc, m - 1\}\);
                    \item \(\sum\limits_{i = 1}^{m - 1} d_{\Gamma}(z_{2, i}, z_{2, i + 1}) \le \sum\limits_{i = 1}^{m - 1} d_{\Gamma}(z_{1, i}, z_{1, i + 1}) = d_{\Gamma}(x_{1}, x_{1}')\).
                \end{mylist}
                Notice that, in general, \(z_{2, m} \ne x_{2}'\). It follows from these conditions and \zcref{step:minactlip3a,step:minactlip3b} that there is a constant \(\ell'\) such that
                \begin{align*}
                    \left\lvert \Phi(x_{1}, t_{1}, x_{2}, t_{2}) - \Phi \left( x_{1}', t_{1}, x_{2}', t_{2} \right) \right\rvert & \le \sum_{i = 1}^{m - 1} \lvert \Phi(z_{1, i}, t_{1}, z_{2, i}, t_{2}) - \Phi(z_{1, i + 1}, t_{1}, z_{2, i + 1}, t_{2}) \rvert \\
                    & \mathrel{\hphantom{\le}} + \left\lvert \Phi \left( x_{1}', t_{1}, z_{2, m}, t_{2} \right) - \Phi \left( x_{1}', t_{1}, x_{2}', t_{2} \right) \right\rvert \\
                    & \le \ell' \sum_{i = 1}^{m - 1}(d_{\Gamma}(z_{2, i}, z_{2, i + 1}) + d_{\Gamma}(z_{1, i}, z_{1, i + 1})) + \ell' d_{\Gamma} \left( z_{2, m}, x_{2}' \right) \\
                    & \le \ell' \left( 2d_{\Gamma} \left( x_{1}, x_{1}' \right) + d_{\Gamma} \left( x_{2}, x_{2}' \right) \right),
                \end{align*}
                where the last inequality is a consequence of
                \begin{equation*}
                    d_{\Gamma} \left( z_{2, m}, x_{2}' \right) \le d_{\Gamma} \left( x_{2}, x_{2}' \right) + \sum_{i = 1}^{m - 1} d_{\Gamma}(z_{2, i}, z_{2, i + 1}) \le d_{\Gamma} \left( x_{2}, x_{2}' \right) + d_{\Gamma} \left( x_{1}, x_{1}' \right).
                \end{equation*}
            \end{steproof}
            \item \emph{Conclusion.} \\
            Let \((x_{1}, t_{1}, x_{2}, t_{2}), (x_{1}', t_{1}', x_{2}', t_{2}') \in A_{C}\), \(\ovt_{1} \coloneqq \min \{t_{1}, t_{1}'\}\) and \(\ovt_{2} \coloneqq \min \{t_{2}, t_{2}'\}\), then~\zcref{eq:minactlip5,eq:minactlip6} show that there is an \(\ell \in \Rds\) so that
            \begin{align*}
                \left\lvert \Phi(x_{1}, t_{1}, x_{2}, t_{2}) - \Phi \left( x_{1}', t_{1}', x_{2}', t_{2}' \right) \right\rvert & \le \left\lvert \Phi(x_{1}, t_{1}, x_{2}, t_{2}) - \Phi \left( x_{1}, \ovt_{1}, x_{2}, \ovt_{2} \right) \right\rvert \\
                & \mathrel{\hphantom{\le}} + \left\lvert \Phi \left( x_{1}, \ovt_{1}, x_{2}, \ovt_{2} \right) - \Phi \left( x_{1}', \ovt_{1}, x_{2}', \ovt_{2} \right) \right\rvert \\
                & \mathrel{\hphantom{\le}} + \left\lvert \Phi \left( x_{1}', \ovt_{1}, x_{2}', \ovt_{2} \right) - \Phi \left( x_{1}', t_{1}', x_{2}', t_{2}' \right) \right\rvert \\
                & \le \ell \left( \left\lvert t_{1} - t_{1}' \right\rvert + d_{\Gamma} \left( x_{1}, x_{1}' \right) + d_{\Gamma} \left( x_{2}, x_{2}' \right) + \left\lvert t_{2} - t_{2}' \right\rvert \right).
            \end{align*}
        \end{steproof}
        This concludes the proof.
    \end{proof}

    \begin{cor}\label{minactlsc}
        The map \(\Phi\) is lower semicontinuous.
    \end{cor}
    \begin{proof}
        Let \(T > 0\) and \(\{(x_{1, n}, t_{1, n}, x_{2, n}, t_{2, n})\}_{n \in \Nds} \subset (\Gamma \times [0, T])^{2}\) be a sequence converging to a fixed \((x_{1}, t_{1}, x_{2}, t_{2})\). It is enough to show that
        \begin{equation}\label{eq:minactlsc1}
            \liminf_{n \to \infty} \Phi(x_{1, n}, t_{1, n}, x_{2, n}, t_{2, n}) \ge \Phi(x_{1}, t_{1}, x_{2}, t_{2}).
        \end{equation}
        \zcref[S]{minactlip} implies that \(\Phi\) is continuous at \((x_{1}, t_{1}, x_{2}, t_{2})\) whenever \(t_{1} < t_{2}\), thus we will assume that \(t_{1} = t_{2}\). If \(x_{1} = x_{2}\) we let \(\zeta_{n} \colon [t_{1, n}, t_{2, n}] \to \Gamma\) be a geodesic linking \(x_{1, n}\) to \(x_{2, n}\) with
        \begin{equation*}
            \left\lvert \dot{\zeta}_{n}(t) \right\lvert = \frac{d_{\Gamma}(x_{1, n}, x_{2, n})}{t_{2, n} - t_{1, n}} \qquad \text{for any \(t \in [t_{1, n}, t_{2, n}]\)},
        \end{equation*}
        then, by \zcref{gLagprop},
        \begin{equation*}
            \Phi(x_{1, n}, t_{1, n}, x_{2, n}, t_{2, n}) \ge \int_{t_{1, n}}^{t_{2, n}} \ovvartheta_{T} \left( \left\lvert \dot{\zeta}_{n}(\tau) \right\lvert \right) \dl \tau = (t_{2, n} - t_{1, n}) \ovvartheta_{T} \left( \frac{d_{\Gamma}(x_{1, n}, x_{2, n})}{t_{2, n} - t_{1, n}} \right).
        \end{equation*}
        This proves~\zcref{eq:minactlsc1} when \(x_{1} = x_{2}\) since
        \begin{equation*}
            \liminf_{n \to \infty}(t_{2, n} - t_{1, n}) \ovvartheta_{T} \left( \frac{d_{\Gamma}(x_{1, n}, x_{2, n})}{t_{2, n} - t_{1, n}} \right) \ge 0 = \Phi(t_{1}, x_{1}, t_{1}, x_{1}).
        \end{equation*}
        Now assume that \(x_{1} \ne x_{2}\). In this case
        \begin{equation*}
            \Phi(t_{1}, x_{1}, t_{1}, x_{2}) = \infty,
        \end{equation*}
        hence we have to prove that
        \begin{equation}\label{eq:minactlsc2}
            \liminf_{n \to \infty} \Phi(x_{1, n}, t_{1, n}, x_{2, n}, t_{2, n}) = \infty.
        \end{equation}
        The coercivity of \(L\) yields that for each \(k \in \Rds\) there is a constant \(b_{k}\) such that
        \begin{equation*}
            L(x, q, t) \ge k \lvert q \rvert + b_{k} \qquad \text{for all \((x, q) \in \Trm \Gamma\) and \(t \in [0, T]\)},
        \end{equation*}
        thereby, denoting with \(\xi_{n}\) an optimal curve for \(\Phi(x_{1, n}, t_{1, n}, x_{2, n}, t_{2, n})\),
        \begin{equation*}
            \Phi(x_{1, n}, t_{1, n}, x_{2, n}, t_{2, n}) \ge k \int_{t_{1, n}}^{t_{2, n}} \left\lvert \dot{\xi}_{n}(\tau) \right\lvert \dl \tau + (t_{2, n} - t_{1, n}) b_{k} \ge kd_{\Gamma}(x_{1, n}, x_{2, n}) + (t_{2, n} - t_{1, n}) b_{k}.
        \end{equation*}
        Notice that, up to subsequences, there is a \(\delta > 0\) so that \(d_{\Gamma}(x_{1, n}, x_{2, n}) \ge \delta\) for every \(n \in \Nds\); therefore
        \begin{equation*}
            \liminf_{n \to \infty} \Phi(x_{1, n}, t_{1, n}, x_{2, n}, t_{2, n}) \ge k \delta,
        \end{equation*}
        which yields~\zcref{eq:minactlsc2} because \(k\) is arbitrary.
    \end{proof}

    \begin{rem}\label{equivLagprops}
        While the results in this section are proved for the overall Lagrangian \(L\), they hold true for any map as in the statement of \zcref{gLagprop}, e.g., the local Lagrangians \(L_{\gamma}\).
    \end{rem}

    \appendix

    \section{Hamiltonians on a Bounded Interval}\label{sec:Honbounded}

    Let \(\ovH \colon [0, 1] \times \Rds^{+} \times \Rds \to \Rds\) be a Hamiltonian satisfying \zcref[comp]{condHcont,condHconv,condHcoer,condHtlip} and define the Lagrangian \(\ovL\) as its convex conjugate. In this \zcref[noref,nocap]{sec:Honbounded} we provide some technical results regarding these maps which are required for our analysis.

    \begin{lem}\label{selection}
        There is a \(\mu \in \Lrm^{\infty}_{loc}(\Rds^{+})\) such that
        \begin{equation}\label{eq:selection.1}
            \ovH(0, t, \mu(t)) = - \ovL(0, 0, t) \qquad \text{for a.e.\ \(t \in \Rds^{+}\)}.
        \end{equation}
    \end{lem}
    \begin{proof}
        Denoting with \(\difcp[-]{\ovL(s, \lambda, t)}{\lambda}\) the subdifferential of \(\lambda \mapsto \ovL(s, \lambda, t)\), we know from~\cite[Proposition~2.22]{Clarke13} that, for each \(s \in [0, 1]\), \(\lambda \in \Rds\), \(t \in \Rds^{+}\) and \(\rho \in \Rds\), its support function is given by
        \begin{equation*}
            \ovL_{\lambda}(s, \lambda, t; \rho) \coloneqq \inf_{h > 0} \frac{\ovL(s, \lambda + h \rho, t) - \ovL(s, \lambda, t)}{h}.
        \end{equation*}
        For any \(\rho \in \Rds\) the map \(t \mapsto \ovL_{\lambda}(0, 0, t; \rho)\) is measurable and locally bounded since it is an infimum of measurable and locally bounded functions, thus~\cite[Proposition~6.29 and Corollary~6.23]{Clarke13} yield that there exists a locally bounded measurable map \(\mu \colon \Rds^{+} \to \Rds\) such that
        \begin{equation*}
            \mu(t) \in \difcp[-]{\ovL(0, 0, t)}{\lambda} \qquad \text{for a.e.\ \(t \in \Rds^{+}\)}.
        \end{equation*}
        By standard properties of convex conjugate functions, this is equivalent to~\zcref{eq:selection.1}.
    \end{proof}

    The next results are crucial for the proof of \zcref{adcurve}.

    \begin{lem}\label{c1sol}
        Given \(t > 0\), \(\varepsilon \in (0, t)\) and \(\delta \in (0, 1)\), let
        \begin{equation}\label{eq:c1sol.1}
            a \ge \max_{\substack{s \in [0, \delta], \\ r \in [t - \varepsilon, t]}} \min_{\mu \in \Rds} \ovH(s, r, \mu).
        \end{equation}
        Then there is a Lipschitz continuous function \(W_{a} \colon [0, \delta] \to \Rds\) such that
        \begin{equation*}
            \max_{r \in [t - \varepsilon, t]} \ovH(s, r, \difcp{W_{a}(s)}{}) = a \qquad \text{for all \(s \in (0, \delta)\)}.
        \end{equation*}
    \end{lem}
    \begin{proof}
        We define the support function
        \begin{equation*}
            \sigma_{a}(s) \coloneqq \max \left\{ \mu \in \Rds : \max_{r \in [t - \varepsilon, t]} \ovH(s, r, \mu) \le a \right\},
        \end{equation*}
        which is a bounded real function in \([0, \delta]\) by~\zcref[noname]{eq:c1sol.1,condHcoer}. The map \(\sigma_{a}\) is clearly upper semicontinuous, thus the sought function is
        \begin{equation*}
            W_{a}(s) \coloneqq \int_{0}^{s} \sigma_{a}(\tau) \dl \tau \qquad \text{for \(s \in [0, \delta]\)}.
        \end{equation*}
    \end{proof}

    \begin{lem}\label{nozenoaux}
        Given \(T > 0\) and a \(\kappa\)-Lipschitz continuous curve \(\eta \colon [0, T] \to [0, 1]\), let \(\varepsilon\) be a positive constant and \(c \colon [0, T] \to \Rds\) be a continuous map such that
        \begin{equation*}
            c(t) \le \min_{\substack{s \in \left[ 0, \frac{\kappa \varepsilon}{2} \right], \\ r \in [t - \varepsilon, t + \varepsilon] \cap [0, T]}} \ovL(s, 0, r) \qquad \text{for all \(t \in [0, T]\)}.
        \end{equation*}
        Then, for any \(t_{1}, t_{2} \in [0, T]\) with \(0 < t_{2} - t_{1} \le \varepsilon\) and \(\eta(t_{1}) = 0 = \eta(t_{2})\),
        \begin{equation*}
            \int_{t_{1}}^{t_{2}} \ovL(\eta, \dot{\eta}, \tau) \dl \tau \ge \int_{t_{1}}^{t_{2}} c(\tau) \dl \tau.
        \end{equation*}
    \end{lem}
    \begin{proof}
        Let \(t_{1}\) and \(t_{2}\) be as in the statement and set
        \begin{equation*}
            a \coloneqq \max_{\substack{s \in \left[ 0, \frac{\kappa \varepsilon}{2} \right], \\ r \in [t_{2} - \varepsilon, t_{2}] \cap [0, T]}} \min_{\mu \in \Rds} \ovH(s, r, \mu),
        \end{equation*}
        then, for any \(t \in [t_{1}, t_{2}]\),
        \begin{equation}\label{eq:nozenoaux1}
            c(t) \le \min_{\substack{s \in \left[ 0, \frac{\kappa \varepsilon}{2} \right], \\ r \in [t - \varepsilon, t + \varepsilon] \cap [0, T]}} \ovL(s, 0, r) \le - \max_{\substack{s \in \left[ 0, \frac{\kappa \varepsilon}{2} \right], \\ r \in [t_{2} - \varepsilon, t_{2}] \cap [0, T]}} \min_{\mu \in \Rds} \ovH(s, r, \mu) = - a.
        \end{equation}
        We point out that the \(\kappa\)-Lipschitz continuity of \(\eta\), together with the hypothesis \(\eta(t_{1}) = 0 = \eta(t_{2})\), implies
        \begin{equation}\label{eq:nozenoaux2}
            \eta([t_{1}, t_{2}]) \subseteq \left[ 0, \frac{\kappa \varepsilon}{2} \right].
        \end{equation}
        By \zcref{c1sol,eq:nozenoaux2} there is a Lipschitz continuous \(W_{a} \colon \left[ 0, \frac{\kappa \varepsilon}{2} \right] \to \Rds\) such that
        \begin{align*}
            \int_{t_{1}}^{t_{2}} \ovL(\eta, \dot{\eta}, \tau) \dl \tau & = \int_{t_{1}}^{t_{2}} \max_{\mu \in \Rds} \left\{ \mu \dot{\eta}(\tau) - \ovH(\eta(\tau), \tau, \mu) \right\} \dl \tau \\
            & \ge \int_{t_{1}}^{t_{2}} \left( \difcp{W_{a}(\eta(\tau))}{} \dot{\eta}(\tau) - \ovH(\eta(\tau), \tau, \difcp{W_{a}}{}) \right) \dl \tau \ge \int_{t_{1}}^{t_{2}}(\difcp{W_{a}(\eta(\tau))}{} \dot{\eta}(\tau) - a) \dl \tau,
        \end{align*}
        thereby the change of variable \(r = \eta(\tau)\) (see~\cite[Corollary~7]{SerrinVarberg69}) and~\zcref{eq:nozenoaux1} yield
        \begin{equation*}
            \int_{t_{1}}^{t_{2}} \ovL(\eta, \dot{\eta}, \tau) \dl \tau \ge \int_{\eta(t_{1})}^{\eta(t_{2})} \difcp{W_{a}(r)}{} \dl r + \int_{t_{1}}^{t_{2}} - a \dl \tau \ge \int_{t_{1}}^{t_{2}} c(\tau) \dl \tau.
        \end{equation*}
        Since \(t_{1}\) and \(t_{2}\) are arbitrary, this proves the claim.
    \end{proof}

    Now consider the equation
    \begin{equation}\label{eq:locHaux}\tag{\={H}J}
        \difcp{U(s, t)}{t} + \ovH(s, t, \difcp{U}{s}) = 0 \qquad \text{in \((0, 1) \times (0, \infty)\)}.
    \end{equation}

    \begin{lem}\label{subsolloc}
        A continuous function \(W \colon (0, 1) \times (0, \infty) \to \Rds\) is a subsolution to~\zcref{eq:locHaux} if and only if, for every interval \([t_{1}, t_{2}] \subset \Rds^{+}\) and Lipschitz continuous curve \(\eta \colon [t_{1}, t_{2}] \to (0, 1)\), one has
        \begin{equation}\label{eq:subsolloc.1}
            W(\eta(t_{2}), t_{2}) - W(\eta(t_{1}), t_{1}) \le \int_{t_{1}}^{t_{2}} \ovL(\eta, \dot{\eta}, \tau) \dl \tau.
        \end{equation}
    \end{lem}

    To prove this \zcref[noref]{subsolloc} we need the following result, taken from~\cite[Lemma~4.1]{Siconolfi22}.

    \begin{lem}\label{suptanmon}
        Let \(f \colon [a, b] \to \Rds\) be a continuous function so that
        \begin{equation*}
            \difcp{\psi(t)}{} \le 0 \qquad \text{(resp.\ \(\ge 0\))}
        \end{equation*}
        for any \(t \in (a, b)\) and \(\Crm^{1}\) supertangent \(\psi\) to \(f\) at \(t\). Then \(f\) is nonincreasing (resp.\ nondecreasing) in \([a, b]\). \\
        The same statement holds by replacing supertangents by subtangents.
    \end{lem}

    \begin{proof}[Proof of \zcref{subsolloc}]
        Fix \(\lambda \in \Rds\) and let \(\psi\) be a \(\Crm^{1}\) supertangent to \(W\) at \((s, t) \in (0, 1) \times (0, \infty)\). If~\zcref{eq:subsolloc.1} holds, we can always find a \(\Crm^{1}\) curve \(\eta\) with \(\eta(t) = s\) and \(\dot{\eta}(t) = \lambda\) such that, for any \(\delta\) small enough,
        \begin{equation*}
            \frac{\psi(s, t) - \psi(\eta(t - \delta), t - \delta)}{\delta} \le \frac{W(s, t) - W(\eta(t - \delta), t - \delta)}{\delta} \le \frac{1}{\delta} \int_{t - \delta}^{t} \ovL(\eta, \dot{\eta}, \tau) \dl \tau.
        \end{equation*}
        Sending \(\delta\) to \(0\) we thus get
        \begin{equation*}
            \difcp{\psi(s, t)}{t} + \difcp{\psi(s, t)}{s} \lambda - \ovL(s, \lambda, t) \le 0.
        \end{equation*}
        Since \(\lambda\) is arbitrary we further get
        \begin{equation*}
            \difcp{\psi(s, t)}{t} + \ovH(s, t, \difcp{\psi}{s}) \le 0,
        \end{equation*}
        which proves that \(W\) is a subsolution. \\
        Now assume that \(W\) is a subsolution, i.e., that for every \((s, t) \in (0, 1) \times (0, \infty)\) and \(\Crm^{1}\) supertangent \(\psi\) to \(W\) at \((s, t)\)
        \begin{equation}\label{eq:subsolloc1}
            0 \ge \difcp{\psi}{t}(s, t) + \ovH(s, t, \difcp{\psi}{s}) \ge \difcp{\psi}{t}(s, t) + \difcp{\psi}{s}(s, t) \lambda - \ovL(s, t, \lambda) \qquad \text{for any \(\lambda \in \Rds\)}.
        \end{equation}
        Given an interval \([t_{1}, t_{2}] \subset \Rds^{+}\) and a \(\Crm^{1}\) curve \(\eta \colon [t_{1}, t_{2}] \to (0, 1)\), let \(\omega\) be a modulus of continuity for \(\ovL\) on \([0, 1] \times [t_{1}, t_{2}] \times [\min \dot{\eta}, \max \dot{\eta}]\). Standard results on sub/superdifferentials (see, e.g., \cite{Clarke13}) show that if \(\phi\) is a \(\Crm^{1}\) supertangent to
        \begin{equation}\label{eq:subsolloc2}
            t \longmapsto W(\eta(t), t) - \int_{t_{1}}^{t} \ovL(\eta, \dot{\eta}, \tau) \dl \tau
        \end{equation}
        at some \(\ovt \in (t_{1}, t_{2})\), then, for any \(\varepsilon > 0\), there are \(s_{\varepsilon} \in (0, 1)\), \(t_{\varepsilon}, t_{\varepsilon}' \in (t_{1}, t_{2})\) and a \(\Crm^{1}\) supertangent \(\psi_{\varepsilon}\) to \(W\) at \((s_{\varepsilon}, t_{\varepsilon})\) such that
        \begin{equation*}
            \left\lvert \ovt - t'_{\varepsilon} \right\rvert < \varepsilon, \qquad \left\lvert \eta \left( \ovt \right) - s_{\varepsilon} \right\rvert + \left\lvert \dot{\eta} \left( \ovt \right) - \dot{\eta} \left( t'_{\varepsilon} \right) \right\rvert + \left\lvert \ovt - t_{\varepsilon} \right\rvert < \varepsilon
        \end{equation*}
        and
        \begin{align*}
            \difcp{\phi}{} \left( \ovt \right) & \le \difcp{\psi_{\varepsilon}}{t}(s_{\varepsilon}, t_{\varepsilon}) + \difcp{\psi_{\varepsilon}}{s}(s_{\varepsilon}, t_{\varepsilon}) \dot{\eta} \left( t'_{\varepsilon} \right) - \ovL \left( \eta, \dot{\eta}, \ovt \right) + \varepsilon \\
            & \le \difcp{\psi_{\varepsilon}}{t}(s_{\varepsilon}, t_{\varepsilon}) + \difcp{\psi_{\varepsilon}}{s}(s_{\varepsilon}, t_{\varepsilon}) \dot{\eta} \left( t'_{\varepsilon} \right) - \ovL \left( s_{\varepsilon}, \dot{\eta} \left( t_{\varepsilon}' \right), t_{\varepsilon} \right) + \omega(\varepsilon) + \varepsilon.
        \end{align*}
        It follows from~\zcref{eq:subsolloc1} that
        \begin{equation*}
            \difcp{\phi}{} \left( \ovt \right) \le \difcp{\psi_{\varepsilon}}{t}(s_{\varepsilon}, t_{\varepsilon}) + \ovH(s_{\varepsilon}, t_{\varepsilon}, \difcp{\psi_{\varepsilon}}{s}) + \varepsilon + \omega(\varepsilon) \le \varepsilon + \omega(\varepsilon).
        \end{equation*}
        Since \(\varepsilon\) is arbitrary, we obtain that \(\difcp{\phi}{} \left( \ovt \right) \le 0\) for each \(\ovt \in (t_{1}, t_{2})\) and \(\Crm^{1}\) supertangent \(\phi\) to~\zcref{eq:subsolloc2} at \(\ovt\). This, in turn, verifies~\zcref{eq:subsolloc.1} for \(\eta \in \Crm^{1}\) thanks to \zcref{suptanmon}. The claim for Lipschitz continuous curves is then established via \(\Crm^{1}\) approximation and the dominated convergence Theorem.
    \end{proof}

    \begin{lem}\label{supsolloc}
        Let \(V \colon (0, 1) \times (0, \infty) \to \Rds\) be a continuous function. If for every \((s, t) \in (0, 1) \times (0, \infty)\) and \(\delta > 0\) small enough there is a curve \(\eta \colon [t - \delta, t] \to (0, 1)\) with \(\eta(t) = s\) such that
        \begin{equation}\label{eq:supsolloc.1}
            V(s, t) - V(\eta(t - \delta), t - \delta) \ge \int_{t - \delta}^{t} \ovL(\eta, \dot{\eta}, \tau) \dl \tau,
        \end{equation}
        then \(V\) is a supersolution to~\zcref{eq:locHaux}.
    \end{lem}
    \begin{proof}
        Given a \(\Crm^{1}\) subtangent \(\varphi\) to \(V\) at \((s, t) \in (0, 1) \times (0, \infty)\), \zcref{eq:supsolloc.1} yields that for any \(\delta > 0\) small enough there is a curve \(\eta\) such that, for every \(\mu \in \Rds\),
        \begin{equation}\label{eq:supsolloc1}
            \begin{aligned}
                \frac{\varphi(s, t) - \varphi(\eta(t - \delta), t - \delta)}{\delta} & \ge \frac{V(s, t) - V(\eta(t - \delta), t - \delta)}{\delta} \ge \frac{1}{\delta} \int_{t - \delta}^{t} \ovL(\eta, \dot{\eta}, \tau) \dl \tau \\
                & \ge \frac{1}{\delta} \int_{t - \delta}^{t} \left( \mu \dot{\eta}(\tau) - \ovH(\eta(\tau), \tau, \mu) \right) \dl \tau \\
                & \ge \mu \frac{\eta(t) - \eta(t - \delta)}{\delta} - \frac{1}{\delta} \int_{t - \delta}^{t} \ovH(\eta(\tau), \tau, \mu) \dl \tau.
            \end{aligned}
        \end{equation}
        Next, we take an infinitesimal sequence \(\{\delta_{n}\}_{n \in \Nds}\). By applying the first mean value Theorem for integrals we obtain a sequence \(\{t_{n}\}_{n \in \Nds}\) converging to \(t\) such that
        \begin{equation*}
            \varphi(s, t) - \varphi(\eta(t - \delta_{n}), t - \delta_{n}) = \difcp{\varphi(\eta(t_{n}), t_{n})}{t} \delta_{n} + \difcp{\varphi(\eta(t_{n}), t_{n})}{s}(\eta(t) - \eta(t - \delta_{n})).
        \end{equation*}
        Setting \(\mu = \difcp{\varphi(\eta(t_{n}), t_{n})}{s}\) in~\zcref{eq:supsolloc1} we thus have, for any \(n \in \Nds\),
        \begin{multline*}
            \difcp{\varphi(\eta(t_{n}), t_{n})}{t} + \difcp{\varphi(\eta(t_{n}), t_{n})}{s} \frac{\eta(t) - \eta(t - \delta_{n})}{\delta_{n}} \\
            \ge \difcp{\varphi(\eta(t_{n}), t_{n})}{s} \frac{\eta(t) - \eta(t - \delta_{n})}{\delta_{n}} - \frac{1}{\delta_{n}} \int_{t - \delta_{n}}^{t} \ovH(\eta(\tau), \tau, \difcp{\varphi(\eta(t_{n}), t_{n})}{s}) \dl \tau,
        \end{multline*}
        which implies
        \begin{equation*}
            \difcp{\varphi(s, t)}{t} + \ovH(s, t, \difcp{\varphi}{s}) = \lim_{n \to \infty} \difcp{\varphi(\eta(t_{n}), t_{n})}{t} + \frac{1}{\delta_{n}} \int_{t - \delta_{n}}^{t} \ovH(\eta(\tau), \tau, \difcp{\varphi(\eta(t_{n}), t_{n})}{s}) \dl \tau \ge 0.
        \end{equation*}
        This proves the claim since \((s, t)\) and \(\varphi\) are arbitrary.
    \end{proof}

    \begin{rem}\label{stateconstloc}
        Let \(V \colon [0, 1) \times (0, \infty) \to \Rds\) be a continuous function. If for every \(t \in (0, \infty)\) and \(\delta > 0\) small enough there is a curve \(\eta \colon [t - \delta, t] \to [0, 1)\) with \(\eta(t) = 0\) such that
        \begin{equation*}
            V(0, t) - V(\eta(t - \delta), t - \delta) \ge \int_{t - \delta}^{t} \ovL(\eta, \dot{\eta}, \tau) \dl \tau,
        \end{equation*}
        then the same arguments used in the proof of \zcref{supsolloc} show that
        \begin{equation*}
            \difcp{\varphi(0, t)}{t} + \ovH(0, t, \difcp{\varphi(0, t)}{s}) \ge 0
        \end{equation*}
        for any constrained \(\Crm^{1}\) subtangent \(\varphi\) to \(V\) at \((0, t)\).
    \end{rem}

    \printbibliography[heading=bibintoc]

\end{document}